\documentclass[english,12pt]{article}
\usepackage{babel}
\usepackage[utf8]{inputenc}
\usepackage{enumerate,graphicx,fontenc,geometry}
\usepackage{amsmath,amsthm,amssymb,stmaryrd,sansmath,color,geometry,multicol,graphicx, float,color,authblk, hyperref,enumerate}

\newtheorem{thm}{Theorem}
\newtheorem{prop}[thm]{Proposition}

\newtheorem{lem}[thm]{Lemma}
\newtheorem{cor}[thm]{Corollary}
\newtheorem{hyp}{Assumption}

\newcommand{\po}{\left(}
\newcommand{\pf}{\right)}
\newcommand{\co}{\left[}
\newcommand{\cf}{\right]}

\newcommand{\R}{\mathbb R}
\newcommand{\E}{\mathbb E}

\newcommand{\D}{\mathcal D}
\newcommand{\dd}{\text{d}}

\newcommand{\na}{\nabla}
\newcommand{\Lv}{L_\varepsilon^*}
\newcommand{\Cinf}{\mathcal C^\infty_+(\R^{2d})}

\title{Hypocoercivity in metastable settings and kinetic simulated annealing.}
\author{Pierre Monmarch\'e\footnote{\emph{Affiliation:} Universit\'e Paris-Est, CERMICS (ENPC), INRIA, 	8 Avenue Blaise Pascal, Cité Descartes,
77455 Marne-la-Vall\'ee, France; \emph{contact:} pierre.monmarche@ens-cachan.org, 06 66 80 50 36 
}}
\date{July 28, 2017\footnote{first version released in March 2015.}}

\begin{document}

\maketitle

\begin{abstract}
Combining classical arguments for the analysis of the simulated annealing algorithm with the more recent hypocoercive method of distorted entropy,  we prove the convergence for large time of the kinetic Langevin annealing with logarithmic cooling schedule.
\end{abstract}
\section{Main result}

Consider the kinetic Langevin diffusion on $\R^{2d}$, which is the solution of the stochastic differential equation
\begin{eqnarray}\label{EqLangevinMasse}
& & \left\{\begin{array}{rcl}
\dd X_t & = & Y_t \dd t\\
& & \\
m\dd Y_t & = & -\na_x U(X_t)\dd t - \nu Y_t \dd t + \sqrt{2 T} \dd B_t,
\end{array}\right.
\end{eqnarray}
where $U$ is a smooth confining potential on $\R^d$ (\emph{confining} meaning that it goes to $+\infty$ at infinity), $m$ a mass, $\nu$ a friction coefficient, $T$ a temperature and $(B_t)_{t\geq 0}$ a standard Brownian motion on $\R^d$. It is ergodic so that,  for $t$ large enough, the law $\mathcal L(X_t,Y_t)$ approximates its equilibrium, which is the Gibbs law with density proportional to $\exp\po -\frac1T\po U(x) + \frac{|y|^2}{2}\pf\pf$ if $m=\nu=1$. At low temperature (namely when $T$ goes to 0) the mass of the Gibbs law concentrates on any neighbourhood of the global minima of $U$. The principle of the annealing procedure is that, if $T$ decays slowly enough with time so that $\mathcal L(X_t,Y_t)$ is still  a good approximation of the Gibbs law in large time, then the process should  reach the global minima of $U$.

This mechanism has been abundantly studied for another process, the stochastic gradient descent
\begin{eqnarray}\label{EqGradientSto}
\dd Z_t & = & - \na U(Z_t) \dd t + \sqrt{2 T_t} \dd B_t,
\end{eqnarray}
which may be obtained from \eqref{EqLangevinMasse} when the mass vanishes or, up to a proper time rescaling, when the friction coefficient $\nu$ goes to infinity (and thus it is also called the overdamped Langevin process). In particular it is known (see \cite{CHSrecuit,Holley} for instance) that there exists a constant $E_*$, depending on $U$ and called the critical depth of the potential, such that, considering a vanishing and positive cooling schedule $(T_t)_{t\geq 0}$, the following holds:
\begin{itemize}
\item if $T_t \geq \frac{E}{\ln t}$ for $t$ large enough with $E>E_*$ then, for all $\delta >0$,
\begin{eqnarray*}
\mathbb P\po U(Z_t) \leq \min U + \delta \pf & \underset{t\rightarrow\infty}\longrightarrow & 1.
\end{eqnarray*}
\item if $T_t \leq \frac{E}{\ln t}$ for $t$ large enough with $E<E_*$ then, for $\delta$ small enough,
\begin{eqnarray*}
\underset{t\rightarrow\infty}\limsup\ \mathbb P\po U(Z_t) \leq \min U + \delta \pf & < & 1.
\end{eqnarray*}
\end{itemize}

The reason for which $Z$ has been more studied than $(X,Y)$ is that it is a reversible process whose \emph{carr\'e du champ} operator is $\Gamma f = |\na f |^2$, which relates its convergence to equilibrium to some functional inequalities satisfied by the Gibbs law (see \cite{LogSob} or Section \ref{SectionGamma} for definitions and more precise statements). On the contrary, $(X,Y)$ is not reversible and $\Gamma f = |\na_y f |^2$ is not elliptic (it lacks some coercivity in the $x$ variable), which is due to the fact the randomness only appears in $\dd Y$ and thus only indirectly intervenes in the evolution of $X$. In other words, $Z$ has been more studied than $(X,Y)$ because,  from a theoretical point of view, it is simpler.

However, from a practical point of view, a process with inertia can be expected to explore the space more efficiently than a reversible one. Indeed, the velocity variable $Y$ acts as an instantaneous memory, which prevents the process to instantaneously go back to the place it just came from. Moreover, the deterministic Hamiltonian dynamics $x''(t) = - \nabla U(x(t)) - x'(t) $ is able to leave the catchment area of a local minimum of $U$, provided it starts with an energy $U(x) + \frac{|x'|^2}{2}$ large enough. This is not the case of the gradient descent $x'(t) = - \nabla U(x(t)) $.

This heuristic, according to which kinetic processes should converge more rapidly than reversible ones, has been proved for some toy models (for instance the Langevin process with a quadratic potential in  \cite{Gadat2013,MonmarcheGuillin}). On the other hand, it has been numerically observed (in \cite{Lelievre2006}) that $(X,Y)$ is, indeed, more efficient than $Z$ (or than the Metropolis-Hastings mutation/selection procedure) in order to sample the Gibbs law at a given temperature for practical potentials. But, to our knowledge, a theoretical proof of the convergence of a simulated annealing algorithm based on the Langevin dynamics was still missing.

\bigskip 

According to the classical analysis of the simulated annealing (developed in the early nineties), the convergence of the algorithm is related to the speed of convergence, at fixed temperature, of the process toward its equilibrium. On the other hand, this question of ergodicity has been intensively investigated over the past fifteen years for degenerated processes such as the Langevin one, which are called hypocoercive.

Bringing together classical arguments (mainly the work of Holley and Stroock \cite{Holley1} and Miclo \cite{Miclo92}) and more recent ideas from  studies of hypocoercivity (mostly the work of Talay \cite{Talay} and Villani \cite{Villani2009}), we will study the convergence of the inhomogeneous Markov process which solves
\begin{eqnarray}
& & \left\{\begin{array}{rcl}\label{EqSDELangevin}
\dd X_t & = & Y_t \dd t\\
& & \\
\dd Y_t & = & -\frac{\theta(\varepsilon_t)}{\varepsilon_t}\na_x U(X_t)\dd t - \frac1{\theta(\varepsilon_t)}Y_t \dd t + \sqrt{2} \dd B_t,
\end{array}\right.
\end{eqnarray}
with positive $\varepsilon_t$ and $\theta$. This is a scaled version of \eqref{EqLangevinMasse} where these two parameters have been chosen so that,  $\varepsilon_t = \varepsilon$ being fixed, the invariant law of the process would be the Gibbs measure associated to the Hamiltonian $ \frac1\varepsilon U(x) + \frac{|y|^2}{2\theta\po \varepsilon\pf}$. Hence, we call $\varepsilon$ the temperature and $(\varepsilon_t)_{t\geq 0}$ the cooling schedule, despite the fact that $\varepsilon$ does not correspond to the physical temperature when the process is interpreted as the position and speed of a particle subjected to potential, friction and thermal forces. Similarly, $\theta$ is the variance of the velocity at equilibrium for a fixed temperature, and we simply call it the variance.

\bigskip

 More precisely, we will work under the following set of hypotheses:

\begin{hyp}\label{Hypo}\

\begin{enumerate}[i] 
\item\label{HypoPotentiel} The potential $U$ is smooth with all its derivatives growing at most polynomially at infinity. It has a finite number of critical points, all of then being non-degenerated (i.e. $U$ is a so-called Morse function), and at least one non-global minimum. Furthermore ,
\[\| \na_x^2 U\|_\infty^2 := \underset{x\in\R^d}\sup\overset{d}{\underset{i,j=1}\sum} \po \partial_{x_i}\partial_{x_j} U(x)\pf^2< \infty\]
and $U$ is quadratic at infinity, in the sense that there exist $a_1,a_2,M,r>0$ such that, for all $x\in\R^d$,
\begin{eqnarray*}
a_1 |x|^2 - M \hspace{20pt}\leq &\ U(x)\ & \leq \hspace{20pt} a_2|x|^2 + M,\\
& & \\
- \na_x U(x).x & \leq & - r|x|^2 +M.
\end{eqnarray*}
\item The cooling schedule is positive, non-increasing, and vanishes at infinity. Moreover, for $t$ large enough, $\partial_t \po \frac{1}{\varepsilon_t}\pf \leqslant \frac{1}{Et} $ where $E>E_*$. In particular, $\varepsilon_t \geqslant \frac{E}{A+ \ln t}$ for some $A>0$, and $|\varepsilon_t'| \leqslant \frac{\varepsilon_0^2}{Et}$.
\item The function $\theta$ is smooth and positive, and $\theta(\varepsilon) \geqslant l\varepsilon$ for some $l>0$. Furthermore, both $\theta$ and $\partial_\varepsilon \theta$ are sub-exponential, where we say a function $\varepsilon \mapsto w(\varepsilon)$ is sub-exponential if $\varepsilon \ln w(\varepsilon)$ goes to 0 as $\varepsilon$ goes to 0.
 
\item\label{HypoLoiInit} The initial law $m_0=\mathcal L(X_0,Y_0)$  admits a $\mathcal C^\infty$ density (still denoted by $m_0$) with respect to the Lebesgue measure. Moreover, the Fisher information $\int \frac{|\na m_0|^2}{m_0} \dd x \dd y$ and the moments $\mathbb E\po |X_0|^p + |Y_0|^p\pf$, $p\geq 0$, are all finite. 
\end{enumerate}

\end{hyp} 

The aim of this work is to prove the following:

\begin{thm}\label{TheoPrincipal}
Under Assumption \ref{Hypo}, if $(X,Y)$ solves \eqref{EqSDELangevin} then, for any $\delta >0$,
\begin{eqnarray*}
\mathbb P\po U(X_t) \leqslant \min U + \delta \pf & \underset{t\rightarrow\infty}\longrightarrow & 1.
\end{eqnarray*}
If, moreover, $\partial_t \po \frac{1}{\varepsilon_t}\pf = \frac{1}{Et} $ for $t$ large enough, then for all $\delta ,\alpha>0$, there exists $A>0$ such that
 \begin{eqnarray*}
\mathbb P\po U(X_t)  \geqslant \min U +    \delta \pf & \leq & A \po\frac1 t\pf^{\frac{\min\po\delta,\frac{E - E_*}{2}\pf}{E}-\alpha}.
\end{eqnarray*}
\end{thm}

\subsection{Organization of the paper}

Some remarks about Theorem \ref{TheoPrincipal} are gathered in Section \ref{SubSectRemarque}, and some numerical examples are provided in Section \ref{SubSectionNumerique}. In Section \ref{SectionSketch}, we give a sketch of the proof of Theorem \ref{TheoPrincipal}, in order to highlight the whole strategy and to precise the technical points which have to be addressed. In particular, it appears that the main question is to study the evolution with time of a so-called distorted entropy. Section \ref{SectionPreliminary} gathers different preliminary considerations, such as the study of the Gibbs measure at small temperature, the existence and smoothness of densities and moment estimates for the kinetic Langevin process. Section \ref{SectionGamma} is devoted to a presentation, in some abstract settings, of the Gamma calculus which is, among other things, a convenient way to compute the evolution of entropy-like terms along a Markov semi-group. The rigorous study of the evolution of the distorted entropy is carried out in Section~\ref{SectionDissipation}, and the proof of Theorem \ref{TheoPrincipal} is concluded in Section \ref{SectionConclusion}.

\subsection{Remarks on Theorem \ref{TheoPrincipal}}\label{SubSectRemarque}

\begin{itemize}
\item The assumption that $U$ is quadratic at infinity  may be seen as an unnecessarily strong requirement, but then the hypocoercive computations are simpler. Anyway, we are mostly concerned with the behaviour of the process in a compact set which contains all the local minima of $U$, since it is the place of the metastable behaviour of the process and thus, of its slow convergence to equilibrium (we refer to \cite{Zitt2008} for some considerations on the growth at infinity of the potential in an annealing framework).
\item The fact that there are two control parameters, $\varepsilon$ and $\theta$, makes the framework slightly more general than the so-called semiclassical studies (cf. \cite{Robbe2014} and references within). These spectral studies furnish precise asymptotics at low (and fixed) temperature of the rate of convergence to equilibrium. However, we will only need very rough estimates since, due to the metastable behaviour of the process, only an exponential large deviation scaling is relevant, and it is given by a log-Sobolev inequality satisfied by the Gibbs law: in other words, it comes from an information on $U$ alone, independent from the Markov dynamics.
\item There are other natural kinetic candidates for the algorithm. We have in mind the run-and-tumble process (see \cite{MonmarcheRTP}, in which the convergence of the annealing procedure is studied), the linear Boltzmann equation (see \cite{Robbe2014} and references within) or the gradient descent with memory (see \cite{GadatPanloup2013}). The reasons for which we considered in a first instance the Langevin dynamics are twofold: first, each of the hereabove processes presents additional difficulties. Both the run-and-tumble process and the Boltzmann one are not diffusions, but piecewise deterministic processes with a random jump mechanism, so that their \emph{carr\'e du champ} is a non-local quadratic operator, satisfying no chain rule. This makes less convenient some forthcoming manipulations on entropies and Fisher informations. As far as the gradient descent with memory is concerned, its invariant measure is not explicit. The second reason to start with the Langevin dynamics is that it has been abundantly studied, so that many results are already available.
\item Our result holds in particular if $\theta(\varepsilon_t) = \varepsilon_t = \frac{E}{\ln(t)}$ with $E>E_*$ for $t>2$. Of course, the critical depth $E_*$ is unknown in practice and, moreover, in a real implementation, the algorithm is only run up to a finite time. In this context, logarithmic cooling schedules are inefficient (see \cite{Catoni} on this topic).
\item Theorem \ref{TheoPrincipal} only yields a sufficient condition for the algorithm to converge, and not a necessary one. In fact, we can't expect any reasonable Markov process whose equilibrium is the Gibbs measure  and with a continuous trajectory (or at least small increments\footnote{Allowing large steps is not a reasonable solution, since in applied problems the dimension is large and the reasonable configurations (i.e. the points where $U$ is not too large) lie in a very small area in view of the Lebesgue measure. A uniformly-generated jump proposal will always be absurd, and rejected.}, such as Gaussian steps for a Metropolis-Hastings algorithm) to allow cooling schedules at a faster order of magnitude than $\frac{E_*}{\ln t}$ (of course, it can be done by an artificial dilatation of the time scale, but this makes no sense in practice). Indeed, heuristically, such a process would take a time of order $\mathcal O(1)$ (at least at an exponential scale) to follow a reaction path, namely to go from a ball around a local minimum to a ball around another minimum without falling back to the first ball. Since, by ergodicity, the ratio between the mean time spent in the reaction path and the mean time spent in the small balls should be of the same order as the ratio between their probability density with respect to the Gibbs law, the mean time between two crossings from one ball to another should be of order $\mathcal O(e^{\frac{E}{\varepsilon}})$ where $E$ is the energy barrier to overcome along the reaction path (this is the so-called Arrhenius law). While the process stays in the catchment area of a local minimum, it   makes successive attempts to escape which, from mixing properties due to the lack of (long-term) memory of the dynamics, are more or less independent one from the other. The time between two decorrelated attempts would be somehow of order $\mathcal O(1)$, which means the probability for each escape attempt to succeed should be of order $\mathcal O(e^{-\frac{E}{\varepsilon}})$, or at least its logarithm should be equivalent to $-\varepsilon^{-1} E$ (this is a large deviation scaling). If $\varepsilon_k$, the temperature at the $k^{\text{th}}$ attempt, is of order $\frac{c}{\ln k}$, then the logarithm of the probability $p_k$ for the $k^{\text{th}}$ attempt to succeed should be of order $-\frac{E}{c}\ln k$, so that $\sum p_k < \infty$ if $c<E$ and $\sum p_k = \infty$ if $c>E$. According to the Borel-Cantelli Theorem, it means the process will almost surely leave the local minimum at some point if $c>E$ (slow cooling), or on the contrary can stay trapped forever with a non-zero probability if $c<E$ (fast cooling). Having a non-zero probability to get trapped forever in the cusp of any local minimum of depth at least $E$ (where the depth of a local minimum $x_0$ is the smallest energy barrier the process has to overcome, starting from $x_0$, in order to reach another minimum $x_1$ with $U(x_1) < U(x_0)$, and the cusp of $x_0$ is the set of points that the process can reach from $x_0$ while staying at an energy level lower than $U(x_0) + E$), it will almost surely end up trapped in one of this cusp. Since $E_*$ is by definition the largest depth among all non-global minima of $U$, if $E>E_*$ then, necessarily, when the process is trapped, it is in the cusp of a global minimum; while if $E<E_*$, it may be in the cusp of a non-global minimum with positive probability, which means the algorithm may fail.

\item Theorem \ref{TheoPrincipal} does not provide any efficiency comparison between the kinetic annealing and the reversible one. That being said, from the previous remark, such a comparison cannot be expected at this level (low temperature and infinite time asymptotic; convergence in probability to any neighbourhood of global minima) for different Markov processes. In fact in practice non-Markov strategies are developed, using memory such as the Wang-Landau or adaptive biasing force (ABF) algorithms (see \cite{LelievreABF,LelievreWL} and references within), or interactions (see \cite{GarciaSun2014}). These dynamics may be Markovian in an augmented space, but the local exploring particle alone is not (and the invariant measure of the augmented Markov process is not the Gibbs measure), so that it may not be limited to a speed $\frac{E_*}{\ln t}$. Since the study \cite{LelievreABF} of the ABF algorithm 
relies on an entropy method, we can hope the present method to extend to this non-Markovian case.
\item The result does not give any indication of what a good choice of $\theta$ could be. It is not even clear whether, as $\varepsilon$ goes to zero, it should go to zero (this is reasonable), to infinity or to a finite positive value. A large $\theta$ allows high velocities, which means a stronger inertia and shorter exit times from local cusps, but may lead in high dimension to the same problem as uniform random large jumps, namely a blind tend to visit absurd configurations, and oscillations between high levels of potential,  hardly affected by too short straight-line crossings of the compact set where all the minima are located. This is reminiscent to the fact that too much memory yields instability for the gradient descent \cite{GadatPanloup2013}.

Since the first submission of the present article, this question has been addressed in the quadratic case in \cite{MonmarcheGuillin}, in which it is proven that, at a fixed $\varepsilon$, the speed of convergence is is optimal for a given intermediary variance $\theta$.
\item Despite the above temperate remarks on its practical interest, we repeat and emphasize that even a theoretical result such as Theorem \ref{TheoPrincipal} was yet to be rigorously established.
\end{itemize}

\subsection{Numerical illustrations}\label{SubSectionNumerique}

A real numerical study, with meaningful elements of comparison between the kinetic and overdamped dynamics, or different choices for the cooling schedule and the variance, on a real optimization problem, would require another paper on its own, as this was done in \cite{Lelievre2006} for the sampling problem at a fixed temperature. Here we only provide some illustrations of our theoretical result.

We consider a toy problem on the one dimensional torus $\R/(2\pi\mathbb Z)$, with the potential
\begin{eqnarray*}
 U(x) & =& \cos(2x)+\frac12\sin(x)+\frac13\sin(10x),
\end{eqnarray*}
which is represented in Fig \ref{FigurePotentiel}.

\begin{figure}[h!]
\centering
\includegraphics[scale=0.35]{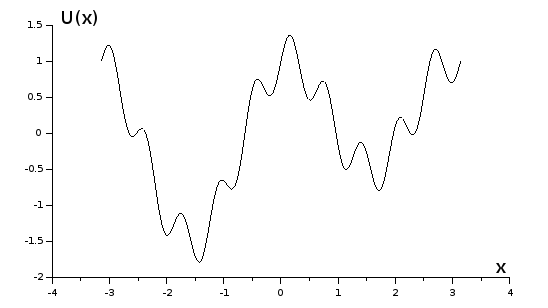}
\caption{A double-well potential with additional small local minima.}\label{FigurePotentiel}
\end{figure}

\begin{figure}[h!]
\centering
\includegraphics[scale=0.35]{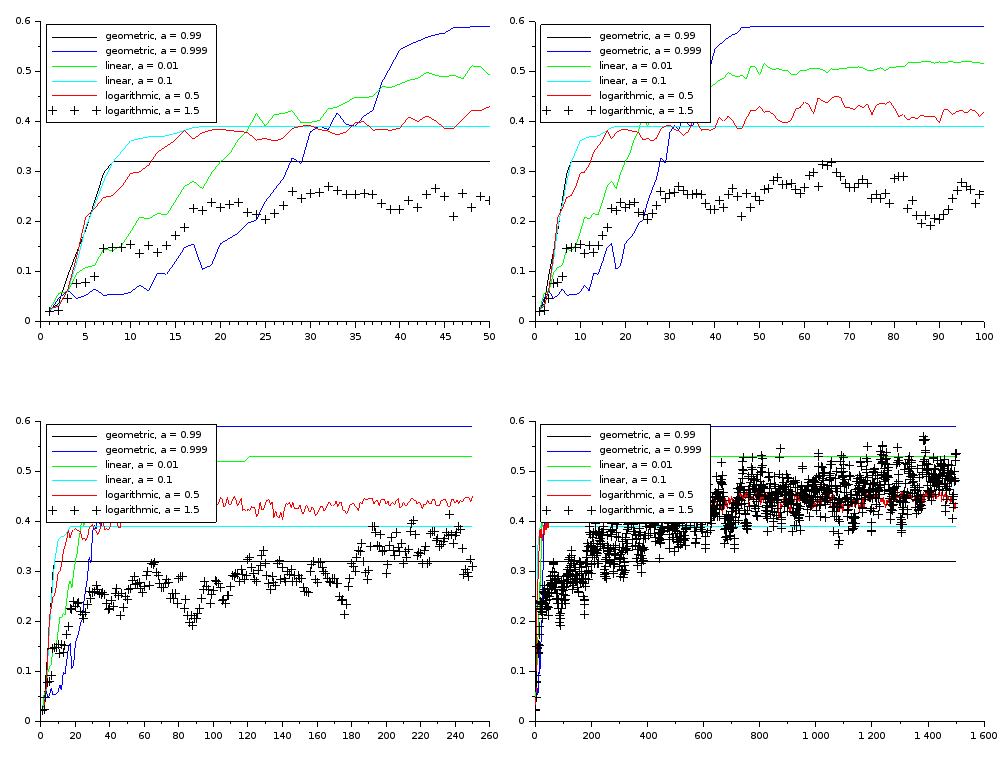}
\caption{An approximation of $t\mapsto \mathbb P \po U(X_t) \leqslant \min U + 0.25\pf$ for different time schedules. The four graphs are the same, except for the time window which is expanded. On the x-axis, each point represent in fact 100 steps of the Euler scheme. The line represented with crosses is the only case that satisfies the condition of Theorem \ref{TheoPrincipal}.}\label{FigureSimu1}
\end{figure}

\begin{figure}[h!]
\centering
\includegraphics[scale=0.35]{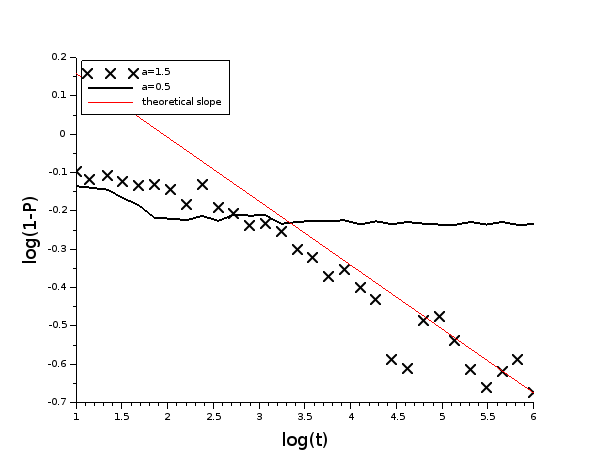}
\caption{A $\log_{10}$-scale representation of the approximation of $t\mapsto 1 - \mathbb P \po U(X_t) \leqslant \min U + 0.25\pf$. The red line is the line  with slope -1/6 which crosses the last point of the case $a=1.5$. Again, each graduation of the x-axis represents 100 Euler scheme steps.}\label{FigureSimu2}
\end{figure}

Rather than of the process \ref{EqSDELangevin}, we run an Euler-Maryuama discretization of the SDE
\begin{eqnarray*}
& & \left\{\begin{array}{rcl} 
\dd X_t & = & Y_t \dd t\\
& & \\
\dd Y_t & = & -\frac{\theta(\varepsilon_t)}{\varepsilon_t}\na_x U(X_t)\dd t -  Y_t \dd t + \sqrt{2 \theta(\varepsilon_t)} \dd B_t,
\end{array}\right.
\end{eqnarray*}
for which the instantaneous invariant measure is still the Gibbs law with temperature $\varepsilon$ and variance $\theta(\varepsilon)$ (the normalization of \ref{EqSDELangevin} has been chosen to lighten some computations, but the proofs of Theorem \ref{TheoPrincipal} can be straightforwardly adapted to this other one). Choosing $\theta(\varepsilon) = \varepsilon$, we compare two different geometric cooling schedule $\varepsilon_k = 5a^k$ with $a = 0.99$ and $0.999$, two linear ones $\varepsilon_k = 5(1+ak)^{-1}$ with $a=0.1$ and $a=0.01$, and two logarithmic ones $\varepsilon_k^{-1} = 1/5 + \ln(1+k)/(aE_*)$ with $a=0.5$ and $a=1.5$ (so that the initial temperature is 5 for all of them). In each case, 100 independent replicas of the process are run and we approximate the probability of success $\mathbb P \po U(X_t) \leqslant \min U + \delta\pf$ by the proportion of replicas for which $U(X_t) \leqslant \min U + \delta$, or more precisely by the mean of this proportion in 100 successive steps. We chose $\delta = 0.25$ so that the set $\{U\leqslant \min U + \delta\}$ is an interval which contains no other minimum than the global one.

The result is represented in Figure \ref{FigureSimu1}. Sooner or later, at some point, all the fast cooling schedule (geometric and linear) freeze, each replica being stuck in a given local minima, so that the proportion of them which are in the correct one stops to evolve after some time. After 12 000 steps of the Euler scheme, only the logarithmic cases still show some variation. But at that time, the probability to be in the right well, in those cases, is around 0.3 and 0.4. It grows with time, but very slowly.

Figure \ref{FigureSimu2} represent, in a $\log_{10}$-scale, the approximation of $t\mapsto \mathbb P \po U(X_t) > \min U + 0.25\pf$ for both logarithmic schedules. After $10^6$ iterations, the case $a=0.5$ stabilizes around the probability 0.4 of success. According to Theorem \ref{TheoPrincipal}, in the case $a=1.5$, the probability of failure should decay at least as $t^{-1/6}$, and Figure \ref{FigureSimu2} suggests that this is indeed the right order of magnitude. After $10^8$ Euler steps, the probability of success is about 0.8. Interpolating along the theoretical slope, it would take around $10^{10}$ Euler steps to reach a probability of success of 0.9.

\section{Sketch of the proof}\label{SectionSketch}

Some notations and the structure of the proof of Theorem \ref{TheoPrincipal} are given here at a formal level. It will be checked in the rest of the paper that the objects introduced here are well-defined and that the arguments can be made rigorous. 

\subsection{Hypocoercivity in the homogeneous case}

Consider an homogeneous Markov process $(Z_t)_{t\geq 0}$ on a Polish space $E$. Let $m_t$ be the law of $(Z_t)$ and let $(P_t)_{t\geq 0}$ and $L$  be, respectively,  the associated semi-group and infinitesimal generator, defined on a suitable set $\D$ of non negative test functions $f$ on $E$ by
\begin{eqnarray*}
P_t f(z) & = &   \E\po f(Z_t)\ |\ Z_0 = z\pf\\
L f & = & \underset{t\rightarrow 0}\lim \frac{P_t f - f}{t}.
\end{eqnarray*}
Suppose that the process admits a unique invariant measure $\mu \in\mathcal P(E)$, where  $\mathcal P(E)$ is the set of probability measures on $E$. Furthermore, suppose that $m_0$ and $P_t$ are such that $m_t \ll \mu$ for all $t\geq 0$. Then, $h_t = \frac{\dd m_t}{\dd \mu}$ is a weak solution of
\begin{eqnarray*}
\partial_t h_t & = & L^* h_t,
\end{eqnarray*}
where $L^*$ is the dual in $L^2(\mu)$ of $L$, defined by $\int f (L g) \dd \mu = \int (L^* f) g \dd \mu$ for $f,g\in\D$. 

For $\nu\in\mathcal P(E)$ and $f\in\D$ with $\int f \dd \nu = 1$, we define the relative entropy $\text{Ent}_{\nu}(f)$ and the Fisher information $I_\nu(f)$ by
\begin{eqnarray*}
\text{Ent}_\nu(f) & = & \int f \ln f \dd \nu ,\\
I_{\nu}(f) & = & \int \frac{|\na f |^2}{f} \dd \nu.
\end{eqnarray*}
We say that $\nu$ satisfies a log-Sobolev inequality with constant $\lambda$ if $ \text{Ent}_\nu(f) \leqslant \lambda I_\nu(f)$ for all $f\in\D$. In the classical case of the overdamped Langevin diffusion (the SDE \eqref{EqGradientSto} with a constant temperature $T$), the exponentially fast convergence of the entropy toward zero (which quantifies the convergence of the law of the process toward its equilibrium) is a direct consequence of the log-Sobolev inequality for $\mu$ and of the Gronwall Lemma, since in that case
\begin{eqnarray*}
\partial_t \po \text{Ent}_\mu\po e^{t L^*} f\pf \pf & = & -I_\mu\po e^{t L^*} f\pf .
\end{eqnarray*}
In the kinetic Langevin case, however, the entropy dissipation is not equal to the full Fisher information, but to a partial one, which may vanish even out of equilibrium. In particular, the entropy doesn't decrease at constant rate. The solution proposed in \cite{Villani2009} is to introduce a distorted entropy
\begin{eqnarray*}
H_\mu(f) & = & \int \frac{|\na_x f +\na_y f|^2}{f}\dd\mu + \gamma \text{Ent}_{\mu }(f),
\end{eqnarray*}
for some $\gamma>0$, whose dissipation along the semi-group will be larger than the Fisher information.

In fact, more generally, we could consider different quantities than those (and, for instance, work with the $L^2$-norm as in \cite{DMS2009}). The main point is that what plays the role of  the entropy, the Fisher Information, the distorted entropy and the distorted entropy dissipation,
\begin{eqnarray*}
K_\mu\po f\pf & := & \po \frac{\dd }{\dd t}\pf_{| t=0}\po H_\mu\po e^{tL^*} f\pf\pf,
\end{eqnarray*} 
should be defined such that, for some positive $C,\rho,\lambda$,
\begin{eqnarray*}
H_\mu(f) & \leq & C \po \text{Ent}_\mu(f) + I_\nu(f) \pf,\\
K_\mu(f) & \leq & - \rho I_\mu(f), \\
\text{Ent}_\mu (f) &  \leq & \lambda I_\mu(f).
\end{eqnarray*}
That way,
\[\partial_t \po H_\mu( e^{tL^*} f) \pf \leq - \frac{\rho}{(1+\lambda)C} H_\mu( e^{tL^*} f),\]
and Gronwall's Lemma yields the exponentially fast decay of $H_\mu$, hence of the entropy.

\bigskip

Now, if $\mu$ is the Gibbs law at temperature $\varepsilon>0$ of some potential with several minima, it is known \cite{MenzSchlichting} that the optimal constant $\lambda$ in the log-Sobolev inequality scales as $\exp(E_*/\varepsilon)$, up to some sub-exponential factors. If we were to prove that the dependence of $\rho$ and $C$ with respect to $\varepsilon$ is sub-exponential, then the leading term of the rate of convergence of $H_\mu$ would be $\lambda^{-1}$, which is exactly the rate of convergence in the overdamped  Langevin case.

\subsection{The non-homogeneous case}\label{SubSectionSketch2}

As a second step, we consider a non-constant temperature. More precisely, let $(\varepsilon_t)_{t\geq0}$ be a non-increasing positive cooling schedule. For $\varepsilon>0$, let $L_\varepsilon$ and $\mu_\varepsilon $ be the generator and invariant measure of a process at fixed temperature $\varepsilon$. We call $L_{\varepsilon_t}$ and $\mu_{\varepsilon_t}$ the instantaneous generator and invariant measure of the inhomogeneous process. Under suitable regularity assumptions, there exists an inhomogeneous Markov process $(Z_t)_{t\geq 0}$ such that the semi-group $(P_{s,t})_{0\leq s\leq t}$ defined by
\[P_{s,t}f(z) = \E\po f(Z_t)\ |\ Z_s=z\pf\]
satisfies $\partial_t P_{s,t} f = P_{s,t} L_{\varepsilon_t} f$ for all $s>0$ and suitable test functions $f\in\D$. We write $L_{\varepsilon}^*$ the dual of $L_{\varepsilon}$ in $L^2(\mu_\varepsilon)$. Let $m_t = \mathcal L(Z_t)$, and suppose that $m_0$ is such that $m_t \ll \mu_{\varepsilon_t}$ for all $t\geq 0$, so that 
\begin{eqnarray}\label{EqEvolm} 
\partial_t \po m_t \pf &  = & \po L_{\varepsilon_t}^*\po \frac{\dd m_t}{\dd \mu_{\varepsilon_t}}\pf\pf\mu_{\varepsilon_t}.
\end{eqnarray}
In other words, writing $h_t = \frac{\dd m_t}{\dd \mu_{\varepsilon_t}}$ and assuming $\frac{\mu'_{\varepsilon_t}}{\mu_{\varepsilon_t}} = \underset{s\downarrow0}\lim \frac1s\po \frac{\dd\mu_{\varepsilon_{t+s}}}{\dd\mu_{\varepsilon_t}} - 1\pf$ exists and is finite,
\begin{eqnarray}\label{EqEvolh}
\partial_t \po h_t\pf & = & L_{\varepsilon_t}^*\po h_t \pf - h_t\frac{\mu'_{\varepsilon_t}}{\mu_{\varepsilon_t}}.
\end{eqnarray}

Denoting $H_t = H_{\mu_{\varepsilon_t}}(h_t)$ and following the previous strategy, a new term appears in the derivative
\begin{eqnarray*}
\partial_t  H_t  & = & K_t + \varepsilon'_t \widetilde K_t,
\end{eqnarray*}
where
\begin{eqnarray*}
K_t \ = \ K_{\mu_{\varepsilon_t}}(h_t) & := & \po \frac{\dd }{\dd s}\pf_{| s=0}\po H_{\mu_{\varepsilon_t}}\po e^{sL_t^*} h_t\pf\pf\\
\widetilde K_t \ = \ \widetilde K_{\mu_{\varepsilon_t}}(h_t) &:=& (\partial_\eta )_{|_{\eta=0}} \po H_{\mu_{{\varepsilon_t}+\eta}}\po h_t \frac{\dd \mu_{{\varepsilon_t}}}{\dd \mu_{{\varepsilon_t}+\eta}} \pf\pf.
\end{eqnarray*}
In the kinetic Langevin case (as in the overdamped one, see \cite{Miclo92}), we will be able to control $|\widetilde K_t |$ with some moments of the process, and to show that these moments' growth with time is slower than any power of $t$. Furthermore $K_t \leqslant - \lambda_t H_t$, where $\lambda_t$ scales like $\exp(-E_*/\varepsilon_t)$, and we will end up with an inequality of the form
\begin{eqnarray*}
\partial_t  H_t  & \leqslant & \po a|\varepsilon_t'|t^{\alpha}  - b t^{-\alpha} e^{-\frac{E_*}{\varepsilon_t}}\pf H_t + c |\varepsilon_t'|  t^\alpha,
\end{eqnarray*}
where $\alpha$ may be chosen arbitrarily small, and $	a,b,c>0$.  The condition on $\varepsilon_t$ ensures that the negative term compensates the others and that $H_t$ goes to zero. By Pinsker's Inequality, this means the total variation distance between $m_t$ the law of the process at time $t$ and its instantaneous equilibrium $\mu_{\varepsilon_t}$ goes to zero. Conclusion follows from the fact that, as $\varepsilon\rightarrow0$, the mass of $\mu_{\varepsilon}$ concentrates on any neighbourhood of the global minima of $U$.

\bigskip

To make rigorous this sketch of proof, to sum up, here is what we have to do in the rest of the paper:
\begin{itemize}
\item Check that everything is well-defined, and that derivation under the integral sign is valid (we will need some truncation arguments).
\item Bound $\widetilde K$ with some moments of the process, and obtain estimates on these moments.
\item Check that the log-Sobolev inequality and the hypocoercive dissipation $K \leqslant - \rho I$ hold.
\item Check that the dependences in $\varepsilon$ are sub-exponential everywhere, except in the log-Sobolev inequality.
\end{itemize}

\subsection{Remarks}

\begin{itemize}
\item For the kinetic Langevin dynamics, the hypocoercive dissipation has been proved by Villani in \cite{Villani2009}. The link between $\widetilde K$ and some  moment of the process appears in the classical annealing analysis of Holley and Stroock \cite{Holley1} and Miclo \cite{Miclo92}. For the kinetic process, moment estimates have been proved at fixed temperature by Talay \cite{Talay} via Lyapunov techniques.
\item The crucial point is, in fact, the dependences with respect to $\varepsilon$. At fixed temperature, the distorted entropy method does not usually yield sharp estimates of the real convergence rate (apart, for now, from the Gaussian case  \cite{Arnold2014,MonmarcheGamma}). Nevertheless, as was already pointed out, since the log-Sobolev constant is exponentially small at low temperature \cite{MenzSchlichting}, only the large deviation scale is relevant, and as long as the other computations stay at a sub-exponential level, they don't need to be sharp.
\item Given a Markov generator $L$ for which a measure $\mu$ is invariant, $(\mu,L)$ is said to satisfy a log-Sobolev inequality with constant $\lambda$ if, for all suitable $f$,
\begin{eqnarray*}
\text{Ent}_\mu(f) & \leqslant & \lambda \int f L\po \ln f\pf \dd \mu.
\end{eqnarray*}
We may call this a Markovian log-Sobolev inequality, in contrast with the classical (or metric) log-Sobolev inequality $\text{Ent}_\mu(f)  \leqslant \lambda I_\mu$. The Markovian inequality depends on a law and some dynamics, while the classical one depends on a law and a gradient, hence a metric. In the classical overdamped langevin case, both coincide, but it is not the case in the kinetic Langevin one, where
\begin{eqnarray*}
  \int f L\po\ln f\pf \dd \mu & = & \int \frac{|\na_y f|^2}{f} \dd \mu,
\end{eqnarray*}
for which a Markovian inequality cannot hold.
\item The fact that the only relevant parameter (the log-Sobolev constant $\lambda$) is just defined from   the measure $\mu$ and from the metric of $\R^d$, and does not depend on the generator, makes sense in view of our previous remark, according to which all reasonable continuous Markov processes devised to sample a Gibbs law should follow the same Arrhenius law, regardless of the local Markov dynamics.
\end{itemize}

\section{Preliminary considerations}\label{SectionPreliminary}

In this section, we consider the diffusion process $(X_t,Y_t)_{t\geq 0}$ on $\R^{2d}$ which solves the SDE 
\eqref{EqSDELangevin}.  The associated generator at temperature $\varepsilon >0$ is
\begin{eqnarray}\label{EqGeneLangevin}
L_\varepsilon & = & y.\na_x - \po\frac{y}{\theta(\varepsilon)} + \frac{\theta(\varepsilon)}{\varepsilon}\na_x U\pf.\na_y + \Delta_y,
\end{eqnarray}
and the corresponding invariant law is 
\[\mu_{\varepsilon}\po \dd x\dd y\pf = Z^{-1}e^{-\frac{U(x)}{\varepsilon}-\frac{y^2}{2\theta(\varepsilon)}}\dd x\dd y,\]
 where $Z = \int e^{-\frac{U(x)}{\varepsilon}-\frac{y^2}{2\theta(\varepsilon)}}\dd x\dd y$ is the normalization constant. 
 We write $\na_x\cdot$ and $\na_y\cdot$ the divergence operators with respect to the variables $x$ and $y$, so that in $L^2\po \mu_\varepsilon\pf$, the dual operators of $\na_x$ and $\na_y$ are
\[\na_x^* = -\na_x\cdot + \frac{1}{\varepsilon}\na_x U\hspace{40pt}\na_y^* = -\na_y\cdot + \frac{1}{\theta\po\varepsilon\pf}y\]
and
\begin{eqnarray}
L_\varepsilon & = & \theta(\varepsilon)\po \na_y^*\na_x - \na_x^*\na_y \pf - \na_y^*\na_y\notag,\\
L^*_\varepsilon & = & \theta(\varepsilon)\po \na_x^*\na_y -\na_y^*\na_x \pf - \na_y^*\na_y\notag\\
 & = & -y.\na_x + \po\frac{\theta(\varepsilon)}{\varepsilon} \na_x U -\frac{y}{\theta(\varepsilon)} \pf.\na_y + \Delta_y.\label{EqGeneLangevin*}
\end{eqnarray}

\subsection{The Gibbs law at small temperature}

As was noted in Section \ref{SubSectionSketch2}, in order to use hypocoercive arguments and prove that a distorted entropy converges to zero, the instantaneous equilibria $\po \mu_\varepsilon\pf_{\varepsilon>0}$ should satisfy a so-called log-Sobolev inequality. 

Under Assumption \ref{Hypo}.\ref{HypoPotentiel}, this inequality is known. More precisely, let $\Cinf$ be the set of smooth positive functions $f$ on $\R^{2d}$, and for $\nu\in\mathcal P\po \R^{2d}\pf$, $f\in\Cinf$ with $\int f \dd\nu = 1$,  let
\begin{eqnarray*}
\text{Ent}_\nu(f) & = & \int f \ln f \dd \nu,\\
I_{\nu}(f) & = & \int \frac{|\na f |^2}{f} \dd \nu
\end{eqnarray*} 
(which are always well-defined, possibly infinite). Recall we say that a function $\varepsilon \mapsto w(\varepsilon)$ is sub-exponential if $\varepsilon \ln w(\varepsilon)$ goes to 0 as $\varepsilon$ goes to 0

\begin{prop}\label{PropLogSob}
Under Assumption \ref{Hypo}.\ref{HypoPotentiel}, there exists $E_*>0$, called the critical depth of $U$, and a sub-exponential $c$ such that, for all  $f\in\Cinf$ with $\int f \dd \mu_\varepsilon=1$,
\begin{eqnarray}\label{EqLogSob}
\emph{Ent}_{\mu_\varepsilon}(f) & \leq & \max\po \frac{\theta(\varepsilon)}{2}, c(\varepsilon) e^{\frac{E_*}{\varepsilon}}\pf I_{\mu_\varepsilon}(f).
\end{eqnarray}
\end{prop}
\begin{proof}
Log-Sobolev inequalities tensorizes (cf \cite{LogSob}) and $\mu_\varepsilon$ is the Cartesian product of its $d$-dimensional marginals (in $x$ and $y$), so that it is sufficient to prove such an inequality for each marginal.

The second marginal is the image by the multiplication by $\sqrt{\theta(\varepsilon)}$ of the standard Gaussian law, which satisfies a log-Sobolev with constant $\frac12$, so that it satisfies a log-Sobolev inequality with constant $\frac{\theta(\varepsilon)}{2}$.

As far as the first marginal is concerned, among several proofs, we refer to the recent work~\cite{MenzSchlichting}. 
\end{proof}
The description of $E_*$ is the following: it is the largest energy barrier the process has to overcome, starting from a non-global minimum of $U$, in order to reach a global one (see \cite{MenzSchlichting}).

\bigskip

According to Section \ref{SubSectionSketch2} again, we also need to show that, as $\varepsilon$ goes to zero, the mass of $\mu_\varepsilon$ concentrates around the global minima of $U$:

\begin{lem}\label{LemMassConcentration}
For any $\delta,\alpha>0$, there exists $c>0$ such that, if $\po \tilde X,\tilde Y\pf$ is a r.v. with law $\mu_\varepsilon$, then, 
\begin{eqnarray*}
\mathbb P \po U\po\tilde X\pf \ > \ \min U + \delta \pf & \leqslant & c e^{-\frac{\delta-\alpha}{\varepsilon}}.
\end{eqnarray*} 
\end{lem}
\begin{proof}
Denote by $vol$ the Lebesgue volume of a Borel set of $\R^{d}$, and, for $\beta>0$,
\begin{eqnarray*}
\mathcal A_\beta &=& \left\{x\in\R^{d},\ U(x) < \min U + \beta \right\}\\
\widetilde {\mathcal A_\beta} &=& \left\{x\in\R^{d},\ \beta < U(x) - \min U < \beta +|x| \right\}.
\end{eqnarray*}
Since $U$ is quadratic at infinity, these sets are compact for any $\beta>0$. Then, we simply bound
\begin{eqnarray*}
\frac{\int \mathbb 1_{x\notin \mathcal A_\delta } e^{-\frac{U(x)}{\varepsilon}}\dd x}{\int  e^{-\frac{U(x)}{\varepsilon}}\dd x} & \leqslant & e^{-\frac{\delta-\alpha}{\varepsilon}}  \frac{vol\po \widetilde {\mathcal A_\delta }\pf + \int e^{-\frac{|x|}{\varepsilon}}\dd x}{vol\po \mathcal A_\alpha\pf}\ \leqslant \  e^{-\frac{\delta-\alpha}{\varepsilon}}  \frac{vol\po \widetilde {\mathcal A_\delta }\pf + \int e^{-|x|}\dd x}{vol\po \mathcal A_\alpha\pf}
\end{eqnarray*}
for $\varepsilon<1$ (the bound is clear for $\varepsilon\geqslant 1$, since a probability is less than 1).
\end{proof}
Note that, rather than the crude bound $\int e^{-\frac1\varepsilon (U-\min U)} \geqslant e^{-\frac\alpha\varepsilon} vol(A_\alpha)$, a Laplace method would give a bound of the right order, $\varepsilon^{-\frac12} e^{-\frac{\delta}{\varepsilon}}$, instead of $e^{-\frac{\delta-\alpha}{\varepsilon}}$.

\subsection{Existence and regularity for the density of the process}


As a first step to make rigorous the computations announced in Section \ref{SubSectionSketch2}, we show in this section that the density of the process \eqref{EqSDELangevin} is nice.

\begin{prop}\label{PropLangevinSens}
Under Assumptions \ref{Hypo}, the process $(X,Y)$ is well-defined for all time, and the second moment $\mathbb E\po |X_t|^2 + |Y_t|^2 \pf$ is finite for all $t>0$. Moreover, $m_t = \mathcal Law(X_t,Y_t)$ admits a smooth density in $\Cinf$ (still denoted by $m_t$) with respect to the Lebesgue measure.

\end{prop}
\begin{proof}
The s.d.e \eqref{EqSDELangevin} admits a solution at least up to the (random) time the process explode to infinity. Consider the homogeneous Markov process $Z_t=(X_t,Y_t,t)_{t\geq 0}$, with generator $\mathcal A = L_{\varepsilon_t} + \partial_t$, and the Hamiltonian 
\[G(x,y,t) = \frac{U(x)-\min U}{\varepsilon_t} + \frac{|y|^2}{2\theta(\varepsilon_t)}+1.\]
 Then, for all $s \leq t$,
\begin{eqnarray*}
\mathcal A G(x,y,s) & = &  \frac{2-|y|^2}{2\theta(\varepsilon_s)} - \varepsilon_s'\po \frac{U(x)-\min U}{(\varepsilon_s)^2} + \theta'(\varepsilon_s)\frac{|y|^2}{2\theta^2(\varepsilon_t)}\pf\\
& \leq & C G(x,y,s),
\end{eqnarray*}
where $C$ depends on the uniform bounds on $[0,t]$ of $\varepsilon$, $\theta$ and their derivatives. Let $\tau_N = \inf\{s>0,\ G(Z_s) > N\}$, so that 
by It\^o's formula, 
\begin{eqnarray}\label{EqMartingale}
\mathbb E\po e^{-Ct\wedge \tau_N} G(Z_{t \wedge \tau_N})\pf & \leq & \mathbb E\po G(Z_{0})\pf 
\end{eqnarray}
\[\Rightarrow \hspace{20pt} \mathbb P\po \tau_N < t \pf \hspace{10pt} \leq \hspace{10pt} \frac{\mathbb E\po  G(Z_{t \wedge \tau_N})\pf}{N} \hspace{10pt} \leq \hspace{10pt} \frac{e^{Ct}\mathbb E\po G(Z_{0})\pf}{N} \hspace{10pt}\underset{N\rightarrow \infty}\longrightarrow 0.\]
Hence, the process is non-explosive and, applying the monotone convergence Theorem in \eqref{EqMartingale}, we get $\mathbb E(G(Z_t)) <\infty$ for all $t$, which implies $\mathbb E\po |X_t|^2 + |Y_t|^2 \pf <\infty$.

The law of the process is a weak solution of the Kolmogorov forward equation \eqref{EqEvolm}, but from \cite{Taniguchi}, since the H\"ormander bracket condition is fulfilled, this equation admits a unique strong solution which is $\mathcal C^\infty$.
\end{proof}

In particular, $h_t = \frac{\dd m_t}{\dd \mu_{\varepsilon_t}}$ is well-defined and smooth. 

\bigskip

\begin{prop}\label{Propm_t>0}
For all $t>0$, $m_t>0$. In particular $h_t$ is bounded below by a positive constant on any compact set.
\end{prop}
\begin{proof}
Following \cite[Lemma 4.2]{GadatPanloup2013}, it is sufficient to prove that the deterministic system associated to the diffusion is approximatively controllable, meaning that for any $z_0,z_1\in\R^{2d}$, for any $\eta >0$ and for any $T>0$, there exists a control $u\in\mathcal C([0,T],\R^d)$ such that the solution $z(t) = \po x(t),y(t)\pf$ of
\[\left\{\begin{array}{ccl}
x'(t) &  = & y(t),\\
y'(t) & =& F_t(z(t)) + u(t),\\
z(0) & = & z_0,
\end{array}\right.\]
where $F_t(x,y) = \frac{\na_x U(x)}{\varepsilon_t} + \frac{y}{\theta(\varepsilon_t)}$, satisfies $|z(T)-z_1|\leq \eta$. Let $z^*(t)$ be the uncontrolled motion, namely the solution of the equation with $u = 0$, and
\[\mathcal O = \left\{s z^*(t) + (1-s) z_1,\ t\in[0,T],\ s\in[0,1]\right\}.\]
Since $\mathcal O \times [0,T]$ is compact, $M = \max\po \|\na_z F_t \|_{L^\infty(\mathcal O \times [0,T])},\max\{y,\ (x,y)\in\mathcal O\}\pf$ is finite. Let $\delta>0$ be small enough. We define $u$ as follow:
\begin{itemize}
\item for $s\in[0,\delta]$, $u(s) = \frac{1}{\delta}\po -y_0 + \frac{x_1-x_0}{T}\pf$,
\item for $s\in[\delta+\delta^2,T-\delta-\delta^2]$, $u(s) = -F_s\po (1-s) x_0 + s x_1,\frac{x_1-x_0}{T}\pf $,
\item for $s\in[T-\delta,T]$, $u(s) = \frac{1}{\delta}\po -\frac{x_1-x_0}{T} + y_1\pf$,
\item for $s\in[\delta,\delta^2]$ and $[T-\delta-\delta^2,T-\delta]$, $u$ is linear.
\end{itemize}
Thus, taking $\delta$  small enough with respect to $M$, $z(\delta)$ and $z(\delta+\delta^2)$ are arbitrarily close to $\po x_0, \frac{x_1-x_0}{T}\pf$, so that $z(T-\delta-\delta^2)$ and $z(T-\delta)$ are arbitrarily close to $\po x_1, \frac{x_1-x_0}{T}\pf$, and $z(T)$ is arbitrarily close to $z_1$.
\end{proof}

\subsection{Moment estimates}

The aim of this subsection is  to prove the following:
\begin{prop}\label{PropMoment}
Under Assumptions \ref{Hypo}, for all $p\in\mathbb N$ and $\alpha>0$,  there exists a constant $k$ such that
\[\E\po \po U(X_t) + |Y_t|^2\pf^p\pf \leq k(1+t)^\alpha.\]
\end{prop}

This result will enable us to control the term denoted as $\widetilde K_t$ in Section \ref{SubSectionSketch2}.

We follow the methods of Talay \cite{Talay} (see also \cite{Wu2001}) and Miclo \cite{Miclo92}, making sure the temperature is only involved in sub-exponential functions.

\begin{lem}\label{LemmeLyapunov1}
Let $\delta^{-1}(\varepsilon) = 4\po 1+ \frac1{\sqrt {a_1l}}\pf\po 1 + \frac{\varepsilon}{2r \theta^3\po\varepsilon\pf}\pf$, and 
\begin{eqnarray*}
R_\varepsilon(x,y) & = & \frac{\theta\po \varepsilon\pf}\varepsilon U(x) +  \frac{|y|^2}{2} + \delta(\varepsilon) x.y.
\end{eqnarray*}
Then, there exist constants $c,C,\rho,N>0$ such that
\[c\po U(x) + |y|^2 \pf - N \hspace{15pt}\leq\hspace{15pt} R_\varepsilon(x,y) \hspace{15pt}\leq \hspace{15pt} C\po \frac{\theta(\varepsilon)}{\varepsilon}U(x) + |y|^2 \pf + N ,\]
and
\begin{eqnarray*}
L_\varepsilon\po R_{\varepsilon} \pf & \leq & - \rho \varepsilon^2 R_{\varepsilon} + N\frac{\theta(\varepsilon)}{\varepsilon}.
\end{eqnarray*}
\end{lem}
\begin{proof}
Since $\frac{\theta\po \varepsilon\pf}{\varepsilon} \geq l$, $U(x)\geq a_1|x^2| -M$ and $\delta(\varepsilon) \leq  \frac1{4}\sqrt {a_1l}$, we get
\[\frac{l}{2} U(x) + \frac{|y|^2}{4} - l M \hspace{15pt}\leq\hspace{15pt} R_\varepsilon(x,y) \hspace{15pt}\leq\hspace{15pt} \po\frac{\theta\po \varepsilon\pf}\varepsilon  + l\pf U(x)+ |y|^2 + lM .\]
Next,  notice that
\begin{eqnarray*}
L_\varepsilon\po \frac{U}{\varepsilon} + \frac{|y|^2}{2\theta\po \varepsilon\pf} \pf & = & - \frac{|y|^2}{\theta(\varepsilon)} + d ,
\end{eqnarray*}
where $d$ is the dimension. On the other hand,
\begin{eqnarray*}
L_\varepsilon \po x.y\pf & = & |y|^2 - \frac{y.x}{\theta\po\varepsilon\pf} - \frac{\theta\po \varepsilon\pf }{\varepsilon}\na_x U(x).x\\
& \leq & \po 1 + \frac{\varepsilon}{2r \theta^3\po\varepsilon\pf} \pf |y|^2 - \frac{\theta\po \varepsilon\pf }{2\varepsilon}\po 2\na_x U(x).x- r |x|^2\pf\\
& \leq & \frac12 \delta^{-1}(\varepsilon) |y|^2 - \frac{r\theta\po \varepsilon\pf }{2a_1\varepsilon}\po U(x) - M\po1+\frac{a_1}{r}\pf\pf.
\end{eqnarray*}
Hence, for some constants $c_i>0$,
\begin{eqnarray*}
L_\varepsilon\po R_{\varepsilon_t}\pf(x,y) & \leq & - \frac12 |y|^2 - c_1 \frac{\delta\po\varepsilon\pf \theta\po\varepsilon\pf}{\varepsilon} U(x) + c_2\frac{\delta\po\varepsilon\pf \theta\po\varepsilon\pf}{\varepsilon} + d\theta\po\varepsilon\pf\\
& \leq & - \delta\po \varepsilon\pf\po c_3 |y|^2 + c_4\po\frac{\theta(\varepsilon)}{\varepsilon} + l \pf U(x)\pf + c_5\frac{\theta(\varepsilon)}{\varepsilon}\\
& \leq & -c_6\delta\po \varepsilon\pf R_\varepsilon(x,y) + c_7\frac{\theta(\varepsilon)}{\varepsilon}.
\end{eqnarray*}
Finally, from $\theta(\varepsilon)\geq l\varepsilon$ and the fact that $\varepsilon$ is non-increasing,
\[\delta^{-1}(\varepsilon) \hspace{15pt}\leq\hspace{15pt} c_8 + \frac{c_9}{\theta^2(\varepsilon)} \hspace{15pt}\leq \hspace{15pt}\frac{c_{10}}{\varepsilon^2}. \]
\end{proof}

\begin{lem}\label{LemmeLyapunov2}
For all $p\in\mathbb N$, there is a sub-exponential $C_p(\varepsilon)$ such that
\begin{eqnarray*}
\partial_\varepsilon \po  R_{\varepsilon}^p \pf & \leq  & C_p\po\varepsilon\pf \po R_{\varepsilon}^p +1\pf.
\end{eqnarray*}
\end{lem}
\begin{proof}
It is straightforward from the fact $\theta,\ \delta,\ \partial_\varepsilon \theta$ and $\partial_\varepsilon \delta$ are sub-exponential functions, since the product and sum of sub-exponential functions are still sub-exponential.
\end{proof}

\begin{lem}\label{LemLyapunov3}
For all $p\in \mathbb N$ and $\alpha>0$ there is a constant $\widetilde C_{p,\alpha}$ such that
\begin{eqnarray*}
\E\co \po R_{\varepsilon_t}(X_t,Y_t)\pf^p\cf & \leq & \widetilde C_{p,\alpha} (1+t)^{1+\alpha}.
\end{eqnarray*}
\end{lem}

\begin{proof}
We prove this by induction. For $p=0$, the result is trivial. Let $p\geq 1$ and suppose that the result holds for all $q< p$. We write $n_{t,p} = \E\co \po R_{\varepsilon_t}(X_t,Y_t)\pf^p\cf$ the $p^{th}$ (distorted) moment at time~$t$. Thanks to Lemma \ref{LemmeLyapunov2},
\begin{eqnarray*}
\partial_t \po n_{t,p}\pf & = & \varepsilon_t' \left.\partial_{\varepsilon}\right|_{\varepsilon = \varepsilon_t}\po \E\co \po R_{\varepsilon}(X_t,Y_t)\pf^p\cf\pf + \partial_s|_{s=0} \po \E\co \po R_{\varepsilon_t}(X_{t+s},Y_{t+s})\pf^p\cf\pf\\
& & \\
& \leq & |\varepsilon_t'|C_p\po\varepsilon_t\pf\po n_{t,p} +1\pf + \E\co \po L_t(R_{\varepsilon_t})^p\pf(X_t,Y_t)\cf.
\end{eqnarray*}
Since $L_{\varepsilon}$ is a second-order derivation operator, for any $\psi\in\mathcal C^\infty \po \R\pf$, $L_{\varepsilon}\po \psi(f)\pf = \psi'(f) L_{\varepsilon} f + \psi''(f) \Gamma_{\varepsilon}(f)$ where $\Gamma_{\varepsilon} = \frac12 L_{\varepsilon} f^2 - f L_{\varepsilon}f$ is the classical \emph{carr\'e du champ} operator. In the case of the kinetic Langevin process, $\Gamma_{\varepsilon} (f) = |\na_y f|^2$. On the other hand, for some $b_i>0$,
\begin{eqnarray*}
|\na_y R_{\varepsilon_t} |^2 & = & |  y + \delta(\varepsilon_t) x|^2\\
& \leq & b_1 R_{\varepsilon_t}  + b_2.
\end{eqnarray*}
Hence, from Lemma \ref{LemmeLyapunov1}, using that $az^{p-1} + bz^{p-2} \leq (a+b)\po z^{p-1} + 1 \pf $ for $p\geq 2$ and $z\geq 0$,
\begin{eqnarray*}
L_{\varepsilon_t}\po R_{\varepsilon_t}^p\pf & \leq & p R_{\varepsilon_t}^{p-1} L_{\varepsilon_t} R_{\varepsilon_t} + p(p-1) (b_1+b_2)\po R_{\varepsilon_t}^{p-1}+1\pf\\
& & \\
& \leq & - \rho p \varepsilon_t^2 R_{\varepsilon_t}^p + b_3 \frac{\theta\po {\varepsilon_t}\pf}{\varepsilon_t}\po R_{\varepsilon_t}^{p-1} + 1\pf.
\end{eqnarray*}
Thus,
\begin{eqnarray*}
\partial_t \po n_{t,p}\pf & \leq & \po |\varepsilon_t'|C_p(\varepsilon_t) - \rho p \varepsilon_t^2\pf n_{t,p} + b_3 \frac{\theta\po {\varepsilon_t}\pf}{\varepsilon_t} (n_{t,p-1}+1) + |\varepsilon_t'|C_p(\varepsilon_t).
\end{eqnarray*}
Let $\alpha>0$. Since 
$\varepsilon_t \geq \frac{E}{A+\ln(1+t)}$, we get $C_p(\varepsilon_t) = \underset{t\rightarrow \infty}o\po t^\beta\pf$ for all $\beta>0$, and the same goes for $\frac{\theta(\varepsilon_t)}{\varepsilon_t}$. By induction, and since $|\varepsilon_t'| \leq \frac{b_4}{t}$, 
\begin{eqnarray*}
\partial_t \po n_{t,p}\pf & \leq &  - b_5 \varepsilon_t^2 n_{t,p}+ b_6 \widetilde C_{p-1,\frac\alpha4} (1+t)^{1+\frac{\alpha}{2}} \\
& & \\
\Rightarrow\hspace{20pt}\partial_t \po n_{t,p} e^{b_5\int_0^t \varepsilon^2_s \dd s}\pf & \leq & b_7(1+t)^{1+\frac{\alpha}{2}} e^{b_5\int_0^t \varepsilon^2_s \dd s}\\
\\
\Rightarrow\hspace{88pt} n_{t,p} & \leq & n_{0,p} +  b_7(1+t)^{1+\frac{\alpha}{2}}\int_0^t e^{-b_5\int_s^t \varepsilon^2_u \dd u}\dd s \\
& \leq &  n_{0,p} +  b_7(1+t)^{1+\frac{\alpha}{2}}\int_0^t e^{-b_5 \varepsilon^2_t (t-s)}\dd s \\
& \leq &  n_{0,p} +  \frac{b_7(1+t)^{1+\frac{\alpha}{2}}}{b_5 \varepsilon^2_t} \\
& \leq & b_8 (1+t)^{1+\alpha}.
\end{eqnarray*}
\end{proof}

\begin{proof}[of Proposition \ref{PropMoment}]
From Lemma \ref{LemmeLyapunov1} and \ref{LemLyapunov3} and Jensen's Inequality, there exist some $a_1,a_2>0$ such that, for any $r\geq 1$,
\begin{eqnarray*}
\E\po \po U(X_t) + |Y_t|^2\pf^p\pf & \leq & a_1 + a_2 \E\po  R^p_{\varepsilon_t}(X_t,Y_t)\pf\\
& \leq & a_1 + a_2 \po\E\po  R^{pr}_{\varepsilon_t}(X_t,Y_t)\pf\pf^{\frac{1}{r}}\\
& \leq & a_1 + a_2 \widetilde C_{1,pr}(1+t)^{\frac{2}{r}},
\end{eqnarray*}
which concludes since $r$ is arbitrarily large.
\end{proof}

\section{Gamma calculus}\label{SectionGamma}

In this section, we are interested, in a formal and general framework, in quantities of the form
\[\Gamma_{L,\Phi}(h) = \frac12 \po L \Phi(h) - D_h \Phi(h).Lh\pf,\]
where $L$ is a Markov operator, $h$ is a non negative function from $\R^d$ to $\R$ which belongs to some functional space $\D$, and $\Phi : \D \rightarrow \R$ is differentiable, with differential operator $D_h \Phi$. Such quantities naturally appear for the following reason: if $L$ admits an invariant measure $\mu$, then $\int L(g) \dd \mu = 0$ for all suitable $g$, so that, writing $h_t = e^{tL}h$,
\[\partial_t \po \int \Phi(h_t) \dd \mu \pf = - \int \po L \Phi(h_t) - D_h \Phi(h_t).Lh_t\pf \dd\mu.\]
In particular, if $\Gamma_{L,\Phi}(h) \geq c \Phi(h)$ for all $h\in\D$ for some $c>0$ then, by the Gronwall's Lemma,
\[ \int \Phi(h_t) \dd \mu \leq e^{-2 ct}\int \Phi(h_0) \dd \mu . \]
There is also a point-wise analogous to this computation: if for $s\in[0,t],$
\[\psi(s) = e^{sL}\po \Phi(h_{t-s})\pf,\]
then
\[\psi'(s) = e^{sL}\po L \Phi(h_{t-s}) - D_h \Phi(h_{t-s}).Lh_{t-s}\pf.\]
If $\Gamma_{L,\Phi}(h) \geq c \Phi(h)$, this yields
\[\Phi(h_t) \leq e^{-2 ct}e^{tL}\Phi(h_0). \]
Integrating with respect to $\mu$ brings  back to $\int \Phi(h_t) \dd \mu \leq e^{-2ct} \int \Phi(h_0)\dd\mu$.

The Gamma calculus is thus a convenient way to retrieve some known computations related to hypocoercivity, in particular the derivation along a semi-group of the distorted entropy, which we will need in Section \ref{SectionDissipation}. We refer to \cite{BolleyGentil,LogSob,BakryGentilLedoux} for an overview of classical Gamma calculus. The latter does not deal with hypoelliptic diffusions, so that we need to expand it in some sense, which is the aim of this section. This will give a new insight on the hypocoercive arguments of Villani \cite{Villani2009}, as has also been proposed by Baudoin in \cite{Baudoin}. Our motivation is to control in a nice way the dependence with respect to the temperature of the estimates we obtain.

Note that, since the first submission of the present work, a more comprehensive presentation of the generalized Gamma presented below has been written in \cite{MonmarcheGamma}. In particular, more general motivations are given, and the novelties with respect to both classical Gamma calculus and existing hypocoercivity results are discussed. This is not the core of the present paper, and thus, we won't go into these details, and only present the computations which will prove useful in Section \ref{SectionDissipation}.

In the remainder of this section, we suppose the space $\D$ is such that everything is well-defined.

\subsection{Quadratic $\Phi$'s}\label{SubSectionQuadra}

First, we consider the case where $\Phi(h) = |A h|^2$ where $A=(A_1,\dots,A_k)$ is a linear operator from $\D$ to $\D^k$. In particular, if $A = I_d$, we retrieve the carr\'e du champ operator which is simply denoted by $\Gamma$:
\[\Gamma(h) \ =\ \Gamma_{L,(.)^2}(h)\ =\ \frac12 L(h^2) - h L h.\]
Another classical example would be $\Phi(h) = |\na h|^2$. For a general quadratic $\Phi$,
\[\Gamma_{L,\Phi}(h)\  =\  \frac12 L|Ah|^2 - (Ah).A(Lh),\]
which directly yields the following:
\begin{lem}\label{LemGammaQuadra}
If $\Phi(h) = |Ah|^2$, then
\begin{eqnarray*}
\Gamma_{L,\Phi}(h) & = & \Gamma(Ah) + (Ah)[L,A]h
\end{eqnarray*}
where for two operator $C$ and $D$, $[C,D]=CD-DC$ and, by convention, we write $\Gamma(Ah) = \sum_{i=1}^k \Gamma(A_i h)$ and $[L,A] = \po [L,A_1],\dots,[L,A_k]\pf$.
\end{lem}
In particular, since $\Gamma$ is always non-negative, $\Gamma_{L,\Phi}(h) \geq (Ah)[L,A]h$.

\bigskip

\textbf{Example 1:} consider the case of a diffusion process with a constant diffusion matrix $B$, namely
\[L h(z) \ = \ b(z).\na h(z) + \na \cdot\po B\na h(z)\pf,\]
where $\na \cdot$ is the divergence operator. Given a constant matrix $M$, set $\Phi(h) = | M \na h |^2$. Then, $\Gamma(h) = \na h.B \na h$ and
\begin{eqnarray*}
\Gamma_{L,\Phi}(h) & \geqslant & (M\na h). [L,M\na] h\\
& = & - (M\na h).M J_b \na h,
\end{eqnarray*}
where $J_b$ is the Jacobian matrix of $b$. Now, suppose that there exists an invertible matrix $M$ for which $-MJ_b(z) M^{-1}$ is bounded below uniformly in $z$ as a quadratic form, meaning that
\[\forall x\in \R^d,\hspace{20pt}-x. MJ_b(z) M^{-1}x \ \geqslant\ \kappa |x|^2\]
 for some $\kappa \in\R$ which does not depend on $z$. Note that this assumption holds in particular if $J_b$ is constant, in which case we retrieve  the results of \cite{Arnold2014}, or, when $B= M^*M$, if the process satisfies a classical Bakry-Emery curvature condition. In that case, writing $h_t = e^{tL} h$, we have
\[|M\na h_t|^2 \ \leqslant \ e^{-2\kappa t} |M\na h_0|^2, \]
and, in particular, $|\na h_t|^2 \leq c e^{-2\kappa t}e^{tL}|\na h_0|^2$ for some $c>0$. Note that such a gradient/semi-group commutation is related to the contraction by $e^{tL}$ of the Wasserstein space $\mathcal W^2$ (see \cite{Kuwada} for further considerations on this topic).

\bigskip

\textbf{Example 2:} consider the case of the kinetic Langevin operator
\begin{eqnarray*}
L & = & -y.\na_x + \po\na_x U(x)-\frac{y}{\theta}\pf.\na_y + \Delta_y.
\end{eqnarray*}
Then, denoting by $\na^2_x U$ the Hessian matrix of $U$,
\begin{eqnarray}\label{EqCommutateur}
\ [L,\na_x]\ = \ -\na^2_x U \na_y,  & \hspace{30pt} &\ [L,\na_y] \ = \  \na_x + \frac{1}{\theta} \na_y .
\end{eqnarray}
If we assume that $\na^2_x U$  is bounded, then $\Gamma_{L,|\na .|^2}(h) \geq - \kappa |\na h|^2$ for $\kappa = \| \na_x^2 U\|_\infty + 1 +\frac{1}{\theta}$, which would only yield, for $h_t = e^{tL}h_0$,
\[|\na h_t|^2\ \leqslant \ e^{2\kappa t} |\na h_0|^2 .\]
Note that this is already more than what would give the Bakry-Emery criterion, since here the Bakry-Emery curvature is equal to $-\infty$.

\bigskip

Now, let $\Phi(h) = |(\na_x + \na_y) h|^2 = |M\na h|^2$, where $M$ is a $d\times(2d)$ matrix constituted of two Identity matrices side by side. In that case, from \eqref{EqCommutateur},
\begin{eqnarray*}
\Gamma_{L,\Phi}(h) & \geqslant & (\na_x+\na_y)h.[L,\na_x+\na_y] h\\
& =& |\na_x h|^2 + \po \frac1\theta - \na_x^2 U(x)\pf |\na_y h|^2 + \po 1+ \frac1\theta -\na_x^2 U(x)\pf\na_y h.\na_x h\\
& \geqslant & |\na_x h|^2 -\| \na_x^2 U\|_\infty  |\na_y h|^2 - \po 1+ \frac1\theta + \| \na_x^2 U\|_\infty \pf \left|\na_y h.\na_x h\right|\\
& \geqslant & \frac12 |\na_x h|^2 - 2 \po\| \na_x^2 U\|_\infty + 1 + \frac1\theta\pf^2 |\na_y h|^2.
\end{eqnarray*}
Writing $\beta = \frac12 + 2 \po\| \na_x^2 U\|_\infty + 1 + \frac1\theta\pf^2$ and $\Psi(h) = \Phi(h) + \beta h^2$, this reads
\[\Gamma_{L,\Psi}\ =\ \Gamma_{L,\Phi}(h) + \beta \Gamma(h)\  \geqslant \  \frac12 |\na h|^2.\]
Up to this point, we have not used the fact that $L$ admits an invariant measure $\mu$; in particular we have not used any information on $\mu$, and we have not split $L$ in its symmetric and anti-symmetric parts in $L^2(\mu)$, as it is usually done in the previous hypocoercive works (such as \cite[Lemma 32]{Villani2009}). The only assumption  on $\mu$ we shall make for now is that it satisfies a Poincar\'e inequality, namely
\[ \int \po h - \int h \dd\mu\pf^2 \dd \mu\leqslant \lambda\int |\na h|^2 \dd \mu \]
for some $\lambda>0$. Then, replacing $h$ by $g = h-\int h\dd \mu$ in the above computation, we get
\[ \int \Psi \po g\pf \dd \mu \ \leqslant \ (2+\beta \lambda) \int |\na g|^2 \dd \mu \ \leqslant \ 2 (2+\beta \lambda) \int \Gamma_{L,\Psi}\po g \pf \dd \mu\  \]
and, by the Gronwall Lemma, since $\int h_t \dd \mu =\int h_0\dd \mu$ for all $t\geqslant 0$,
\[\int \Psi \po h_t - \int h_0 \dd \mu\pf \dd\mu \ \leqslant\ e^{-\frac{t}{2 + \beta \lambda}} \int \Psi \po h_0 - \int h_0 \dd \mu\pf \dd\mu.\]

\subsection{Entropic $\Phi$'s}

In this subsection, $L$ is a diffusion operator:
\[L h \ =\ b \na h + \na \cdot \po B \na h\pf,\]
where $b$ is a vector field and $B$ is a symmetric positive matrix-valued function. This is equivalent to the fact that for any $\psi\in\mathcal C^\infty \po \R\pf$, 
\begin{eqnarray*}
L\po \psi(h)\pf & = & \psi'(h) Lh + \psi''(h) \Gamma(h),\\
L(gf) & = & gL(f)+fL(g) + 2\Gamma(g,f),
\end{eqnarray*}
where $\Gamma(f,g)$ stands for the symmetric bilinear operator associated by polarization to the quadratic operator $\Gamma$. We recall the following classical lemma:
\begin{lem}
If $\Phi(h) = h\ln h$ then
\begin{eqnarray*}
\Gamma_{L,\Phi}(h) & = & \frac{\Gamma(h)}{2 h}.
\end{eqnarray*}
\end{lem}
\begin{proof}
Let $\psi(x) = x \ln x$, so that $\psi'(x) = 1 + \ln x$, $\psi''(x) = \frac1x$, and
\begin{eqnarray*}
2 \Gamma_{L,\Phi}(h) & = & (1+\ln h) Lh + \frac{\Gamma(h)}{h} -(1+\ln h)Lh.
\end{eqnarray*}
\end{proof}
Note that, since $\Gamma$ is the square of a first order differential operator, $\frac{\Gamma(h)}{h} = 4\Gamma\po\sqrt h\pf$. When the diffusion matrix of the generator is $I_d$ and $\mu$ is the invariant measure, $\int \Gamma(\sqrt h) \dd \mu = \int |\na \sqrt h|^2 \dd \mu$ and we retrieve the Fisher Information of $h$ with respect to $\mu$.

\bigskip

Let us consider, for a matrix-valued function $M\in\mathcal C^2\po \R^d,\mathcal M_d(\R)\pf$,
\begin{eqnarray*}
\Phi_M(h) & = & \frac{| M \na  h|^2}h,
\end{eqnarray*}
which may be called a Fisher Information-like term, even if $M$ is not invertible.

\begin{lem}
If $\Phi_M(h)  =  \frac{| M \na  h|^2}h$, then
\begin{eqnarray*}
\Gamma_{L,\Phi_M}(h) &  \geq &   \frac{M\na h.[L,M\na]h}{h} .
\end{eqnarray*}
\end{lem}
\begin{proof}
Making use of the diffusion property of $L$, we compute
\begin{eqnarray*}
\Gamma_{L,\Phi_M}(h) & = &  \frac12L\po \frac{|M\na h|^2}{h} \pf- 4M\na \sqrt{h}.M\na\po \frac{Lh}{2\sqrt h}\pf\\
& = & \frac12\frac{L\po |M\na h|^2\pf}{h} + \frac12|M\na h|^2 L\po \frac1h\pf + \Gamma\po \frac1h,|M\na h|^2\pf\\
& & - \frac{M\na h.M\na Lh}{ h} -  \frac{M\na h . M\na\po h^{-\frac12}\pf}{\sqrt h}Lh\\
& =& \frac{\Gamma_{L,|M\na .|^2}(h)}{h} + \Gamma\po \frac1h,|M\na h|^2\pf + |M\na h|^2\frac{\Gamma(h)}{h^3}.
\end{eqnarray*}
From Lemma \ref{LemGammaQuadra} and the fact that $\Gamma\po \frac1f,g^2\pf = -\frac{\Gamma(f,g^2)}{f^2} =- \frac{2\Gamma(f,g)g}{f^2}$,
\begin{eqnarray*}
\Gamma_{L,\Phi_M}(h) & = &   \frac{\Gamma\po M\na h\pf+M\na h.[L,M\na]h}{h} - \frac{2\Gamma(h,M\na h).M\na h}{h^2}+ |M\na h|^2\frac{\Gamma(h)}{h^3},
\end{eqnarray*}
where by convention $\Gamma(h,M\na h) = \po \Gamma\po h,(M\na h)_1\pf,\dots,\Gamma\po h,(M\na h)_d\pf\pf$. As the diffusion matrix $B$ is symmetric and positive, $B=Q^*Q$ for some real matrix $Q$ and $\Gamma(f) = |Q\na f|^2$, which yields $|\Gamma(f,g)| \leq \sqrt{\Gamma(f)\Gamma(g)}$ and
\begin{eqnarray*}
- \frac{2\Gamma(h,M\na h).M\na h}{h^2} & \geq & -\frac{\Gamma\po M\na h\pf}{h}-|M\na h|^2\frac{\Gamma(h)}{h^3}.
\end{eqnarray*}
\end{proof}

\begin{cor}\label{CorGamma}
Suppose
\begin{eqnarray*}
L & = & -y.\na_x + \po\na_x U(x)-\frac{y}{\theta}\pf.\na_y + \Delta_y
\end{eqnarray*}
with $U$ such that $\|\na_x^2 U\|_\infty < \infty$, and let 
\begin{eqnarray*}
 \beta & = & 1 + 2\po\| \na_x^2 U\|_\infty + 1 + \frac1\theta\pf^2\\
 \Phi(h) & = & \frac{|(\na_x + \na_y) h|^2}{h} + \beta h\ln h,\\
 \Phi_2(h) & = & \frac{|\na h|^2}{h}.
\end{eqnarray*}
Then
\begin{eqnarray*}
\Gamma_{L,\Phi} & \geq & \frac12 \Phi_2\\
\Gamma_{L,\Phi_2} & \geq & - \po \|\na_x^2 U\|_\infty + 1 + \frac{1}{\theta}\pf \Phi_2.
\end{eqnarray*}
\end{cor}
\begin{proof}
All the computations have already been executed in Example 2 of Section \ref{SubSectionQuadra}.
\end{proof}
In Section \ref{SectionDissipation}, this result will be the core of the hypocoercivity dissipation introduced in Section~\ref{SectionSketch}. Note that we followed the ideas of \cite[Lemma 32]{Villani2009}, but in a somehow simpler presentation, so that the dependences with respect to $U$ and $\theta$ is clear, and without using any information about the invariant measure.

\section{Distorted entropy dissipation}\label{SectionDissipation}

We use here the notations of Section \ref{SectionPreliminary}. In particular, $h_t = \frac{m_t }{\mu_{\varepsilon_t}}$, where $m_t$ is the law of the process~\eqref{EqSDELangevin}. We now introduce the distorted entropy
\begin{eqnarray*}
H(t) & = & \int \frac{|\na_x h_t +\na_y h_t|^2}{h_t}\dd\mu_{\varepsilon_t} + \gamma\po \varepsilon_t\pf \text{Ent}_{\mu_{\varepsilon_t}}(h_t),
\end{eqnarray*}
where $\gamma(\varepsilon) = 1 + 2\po \frac{\theta(\varepsilon)}{\varepsilon}\| \na_x^2 U\|_\infty + 1 + \frac1{\theta(\varepsilon)}\pf^2$. The aim of this section is to prove the following:
\begin{prop}\label{PropH}
For all $t>0$, the Fisher information $I(t) = \int \frac{|\na h_t|^2}{h_t}\dd\mu_{\varepsilon_t}$ is finite, and $t\mapsto I(t)$ is locally bounded. Moreover, $H$ is absolutely continuous and there exists a sub-exponential function $\xi$ such that, for almost every $t\geq 0$,
\[H'(t) \leq - I(t)  + |\varepsilon_t'|\xi(\varepsilon_t)\po H(t) + 1+ \mathbb E(X_t^2 + Y_t^2)\pf .\]
\end{prop}
Recall that, according to Section \ref{SubSectionSketch2}, such a differential inequality is the main technical argument in the proof of Theorem \ref{TheoPrincipal}, since it will imply that $H(t)$, hence $\text{Ent}_{\mu_{\varepsilon_t}}(h_t)$, goes to zero as $t$ goes to infinity.

\subsection{Truncated differentiation}

In the first instance, we will consider a truncated version of $H$ in order to differentiate under the integral sign. We write
\begin{eqnarray*}
\Phi_0(h) & = & h \ln h, \\
\Phi_1(h) & = & \frac{|(\na_x + \na_y) h|^2}{h},\\
\Phi_2(h) & = & \frac{|\na h|^2}{h}.
\end{eqnarray*}
These quantities are well-defined for $h\in \mathcal D = \{h\in\mathcal C^\infty(\R^{2d}),\ h>0\}$ and, for any smooth compactly-supported $\eta \in \mathcal C_c^\infty\po \R^{2d}\pf$, so is $\int \eta \Phi_i(h) \dd \mu_{\varepsilon_t}$ for $i=0,1,2$. From Propositions \ref{PropLangevinSens} and \ref{Propm_t>0}, for all $t\geq 0$, $h_t \in \mathcal D$ and so does $m_t = h_t \mu_{\varepsilon_t}$.

  Let $\mathcal T = \{\eta \in \mathcal C_c^\infty(\R^{2d}), 0 \leq \eta \leq 1\}$. We are interested in the truncated distorted entropy and the truncated Fisher information:
\begin{eqnarray*}
H_\eta(t)  & = & \int \eta \po \Phi_1(h_t) + \gamma(\varepsilon_t) \Phi_0(h_t) \pf \dd \mu_{\varepsilon_t}\\
I_\eta(t) & = & \int \eta \Phi_2(h_t)  \dd \mu_{\varepsilon_t}.
\end{eqnarray*}
 In order to understand the time evolution of (the truncated versions of) $H$ and $I$, following the ideas of Section \ref{SubSectionSketch2} (with the distinction between $K_t$ and $\widetilde K_t$), we will distinguish the roles of, on the one hand, the evolution of the temperature at fixed time (Lemma \ref{LemKtilde} below)  and, on the other hand, the convergence to equilibrium at fixed temperature (Lemma \ref{LemK} below).


\begin{lem}\label{LemKtilde}
There exists a sub-exponential function $\xi(\varepsilon)$ such that for all $\eta \in\mathcal T$ and for all $m\in\D$,
\begin{eqnarray*}
\partial_\varepsilon \po \int \eta \Phi_i \po \frac{m}{\mu_{\varepsilon}}\pf \dd \mu_{\varepsilon}\pf & \leqslant & \xi(\varepsilon) \po \int \eta \Phi_i \po \frac{m}{\mu_{\varepsilon}}\pf \dd \mu_{\varepsilon} + \int (1+x^2+y^2) m(x,y) \dd x \dd y\pf
\end{eqnarray*}
for $i=1,2$, and
\begin{eqnarray*}
\partial_\varepsilon \po \gamma(\varepsilon)\int \eta \Phi_0 \po \frac{m}{\mu_{\varepsilon}}\pf \dd \mu_{\varepsilon}\pf & \leqslant & \xi(\varepsilon) \po \gamma(\varepsilon) \int \eta \Phi_0 \po \frac{m}{\mu_{\varepsilon}}\pf \dd \mu_{\varepsilon} + \int (1+x^2+y^2) m(x,y) \dd x \dd y\pf.
\end{eqnarray*}
\end{lem}

\begin{proof}
 Note that $\Phi_0 \po \frac{m}{\mu_{\varepsilon}}\pf  \mu_{\varepsilon} = \ln\po \frac{m}{\mu_{\varepsilon}}\pf m $ and that $\Phi_i \po \frac{m}{\mu_{\varepsilon}}\pf  \mu_{\varepsilon} = \left|M_i\na \ln  \po \frac{m}{\mu_{\varepsilon}}\pf\right|^2 m $ for $i=1,2$ for some matrices $M_1,M_2$. Hence, we compute
\begin{eqnarray*}
\partial_\varepsilon \ln \mu_\varepsilon(x,y)
& = &   \frac{U(x)}{\varepsilon^2}  + \frac{\partial_\varepsilon \theta\po\varepsilon\pf |y|^2}{\theta^2(\varepsilon)} - \frac{\int \po \frac{U(u)}{\varepsilon^2}  + \frac{\partial_\varepsilon \theta\po\varepsilon\pf |v|^2}{\theta^2(\varepsilon)}\pf e^{- \frac{U(u)}{\varepsilon} - \frac{|v|^2}{2\theta(\varepsilon)}} \dd u\dd v} {\int e^{- \frac{U(u)}{\varepsilon} - \frac{|v|^2}{2\theta(\varepsilon)}}\dd u\dd v}\\
& = &  \frac{U(x)}{\varepsilon^2}  + \frac{\partial_\varepsilon \theta\po\varepsilon\pf |y|^2}{\theta^2(\varepsilon)} - \int \po  \frac{U(u)}{\varepsilon^2}  + \frac{\partial_\varepsilon \theta\po\varepsilon\pf |v|^2}{\theta^2(\varepsilon)} \pf \dd \mu_\varepsilon(\dd u,\dd v).
\end{eqnarray*}
The moments of the family $(\mu_\varepsilon)_{\varepsilon < 1}$ are clearly bounded uniformly with respect to $\varepsilon$, so that there exists a sub-exponential $\xi_1$ such that
\[|\partial_\varepsilon \ln \mu_\varepsilon(x,y)| + |\na \partial_\varepsilon \ln \mu_\varepsilon(x,y)|^2 \leqslant \xi_1(\varepsilon) (1+x^2 + y^2).\]
It implies, using that $|\Phi_0(h)| = |\Phi_0(h) + \frac1e - \frac1e|\leq \Phi_0(h) + \frac{2}{e}$ (as $\frac1e = \underset{x \in\R}\inf \po x \ln x\pf$), that
\begin{eqnarray*}
\partial_\varepsilon \po \gamma(\varepsilon) \int \eta \Phi_0 \po \frac{m}{\mu_{\varepsilon}}\pf  \mu_{\varepsilon}\pf & = & \int \eta \partial_\varepsilon \ln \mu_\varepsilon  m + \gamma'(\varepsilon) \int  \eta \Phi_0 \po \frac{m}{\mu_{\varepsilon}}\pf  \mu_{\varepsilon}\\
& \leqslant &  \xi_1(\varepsilon) \int (1+x^2 + y^2) m(x,y) \dd x \dd y + |\gamma'(\varepsilon)|\po \int  \eta  \Phi_0 \po \frac{m}{\mu_{\varepsilon}}\pf \mu_{\varepsilon} + \frac2e \pf
\end{eqnarray*} 
and, for $i=1,2$,
\begin{eqnarray*}
\partial_\varepsilon \po \int \eta \Phi_i \po \frac{m}{\mu_{\varepsilon}}\pf  \mu_{\varepsilon}\pf & = &  -2 \int  M_i \na \ln \po \frac{m}{\mu_\varepsilon} \pf. M_i \na \partial_\varepsilon \ln \mu_\varepsilon \ \dd m \\
& \leqslant & \int \eta \Phi_i \po \frac{m}{\mu_{\varepsilon}}\pf  \mu_{\varepsilon} + 2  \xi_1(\varepsilon) \int (1+x^2 + y^2) m(x,y) \dd x \dd y.
\end{eqnarray*}
Conclusion follows with $\xi = 1 + 2 \xi_1 + \frac{|\gamma'|}{\gamma}$.
\end{proof}

For the next Lemma, we use the notions and notations of Section \ref{SectionGamma}.

\begin{lem}\label{LemK}
Suppose that $\eta \in\mathcal T$ is such that $\po 2\Lv + L_\varepsilon\pf \eta \geqslant - c\eta$ for some $c>0$. Then, for all $h\in\D$,
\begin{eqnarray*}
\partial_t \po \int \eta \Phi_i \po e^{tL_\varepsilon^*} h\pf \dd \mu_{\varepsilon}\pf & \leqslant & c\po \frac1e + \int \eta \Phi_i \po e^{tL_\varepsilon^*} h\pf \dd \mu_{\varepsilon} \pf -  2\int \eta \Gamma_{L_{\varepsilon}^*,\Phi_i} \po e^{tL_\varepsilon^*} h\pf \dd \mu_{\varepsilon}.
\end{eqnarray*}
\end{lem}

\begin{proof}
In this proof, we write $h_t = e^{tL_\varepsilon^*}h$ (note that $\D$ is fixed by $e^{tL_\varepsilon^*}$ from Propositions \ref{PropLangevinSens} and \ref{Propm_t>0} applied to the process with constant temperature schedule). Since the support of $\eta$ is compact, we can differentiate under the integral sign:
\begin{eqnarray*}
\partial_t \po \int \eta \Phi_i \po h_t\pf \dd \mu_{\varepsilon}\pf  & = &  \int \eta D\Phi_i \po h_t\pf. L_\varepsilon h_t  \dd \mu_{\varepsilon} \\
& = &   \int \eta D\Phi_i \po h_t\pf. L_\varepsilon^* h_t  \dd \mu_{\varepsilon} - \int L_\varepsilon^*\po \eta \Phi_i(h_t)\pf \dd\mu_\varepsilon .
\end{eqnarray*}
Since $L_\varepsilon^*$ is a diffusion operator,
\[L_\varepsilon^* \po \eta \Phi_i(h_t) \pf = \eta \Lv \Phi_i(h_t) + \Phi_i(h_t) \Lv \eta +  2 \Gamma_\varepsilon \po \eta, \Phi_i(h_t)\pf,\]
and
\[\int \Gamma_\varepsilon(f,g) \dd \mu_\varepsilon = - \frac12 \int \po f \Lv g + g \Lv f \pf \dd \mu_\varepsilon = -\frac12 \int f \po \Lv + L_\varepsilon\pf g \dd \mu_\varepsilon.\]
Hence,
\begin{eqnarray*}
\partial_t \po \int \eta \Phi_i \po h_t\pf \dd \mu_{\varepsilon}\pf  & = &  -  2\int \eta \Gamma_{L_{\varepsilon}^*,\Phi_i} \po h_t\pf \dd \mu_{\varepsilon} - \int \Phi_i(h_t) \po 2\Lv + L_\varepsilon\pf (\eta) \dd \mu_\varepsilon\\
& = & -  2\int \eta \Gamma_{L_{\varepsilon}^*,\Phi_i} \po h_t\pf \dd \mu_{\varepsilon} - \int \po \Phi_i(h_t) + \frac1e\pf \po 2\Lv + L_\varepsilon\pf (\eta) \dd \mu_\varepsilon.
\end{eqnarray*}
The constant $\frac1e $ is added in order to ensure $\Phi_i(h) + \frac1e\geq 0$ for all $i=0,1,2$.
 We conclude with
\[\int \eta \dd \mu_\varepsilon \leqslant \int \dd \mu_\varepsilon = 1.\]
\end{proof}

\subsection{Construction of the truncation}

We now describe a particular choice of $\eta\in\mathcal T$ which satisfies the assumption of Lemma~\ref{LemK}.

Let $\kappa >0$ and
\[l(v) = \left\{\begin{array}{cll}
1 & \hspace{10pt}& \text{if }v\leq -\frac{\pi}{\kappa},\\
 \frac{-\kappa v - \sin(\kappa v)}{2\pi} + \frac12  & \hspace{10pt}& \text{if }v\in\left[-\frac{\pi}{\kappa},\frac{\pi}{\kappa}\right],\\
0 & \hspace{10pt}& \text{if } v\geq \frac{\pi}{\kappa}.
\end{array}\right.\]
Then $l\in\mathcal C^2(\R)$ is a non-increasing non-negative function with
\[l''(v) = \frac{\kappa^2}{2\pi} \sin(\kappa v) \mathbb 1_{v\in \left[-\frac{\pi}{\kappa},\frac{\pi}{\kappa}\right]} \geq - \frac{\kappa^2}{2\pi} \mathbb 1_{v\in \left[-\frac{\pi}{\kappa},0\right]} \geq - \frac{\kappa^2}{4\pi}l(v).\]
Let $\eta(x,y) = l\po V(x,y) - p \pf$ for $p\in\mathbb N$, with $V$ to be chosen later in order to satisfy the following conditions:
\begin{itemize}
\item $V$ goes to $+\infty$ at $\infty$ (the level sets of $V$ are compact),
\item $V$ is Lipschitz,
\item $V$ is a Lyapunov function for $L_{\varepsilon_t} + 2L_{\varepsilon_t}^*$, in the sense $\po L_{\varepsilon_t} + 2L_{\varepsilon_t}^*\pf V \leqslant 0$ outside a compact.
\end{itemize}
In the first instance, suppose that we have constructed such a function $V$.
Let $\tilde L_t =  2L_{\varepsilon_t}^* + L_{\varepsilon_t}$, and let $p$ be large enough so that the compact $\{\tilde L_t V \geq 0\}$ is included in $\{V \leq p - \frac\pi\kappa\}$. On  $\{V \leq p - \frac\pi\kappa\}$, $\eta = 1$ and thus $\tilde L_t \eta = 0$. On  $\{V \geq p - \frac\pi\kappa\}$, $\tilde L_t V \leq 0$ and thus
\begin{eqnarray*}
\tilde L_t\eta  & = &  l'(V-p) \tilde L_t V + l''(V-p) \Gamma_{\tilde L_t, (.)^2}(V)\\
& \geqslant & 3 l''(V-p)| \na_y V |^2\\
& \geqslant & - \frac{3\kappa^2}{4\pi} \| \na_y V \|^2_\infty \eta.
\end{eqnarray*}
Hence, in order to apply Lemma \ref{LemK}, it only remains to find a Lyapunov function $V$. The problem is that it has to be a Lyapunov function for $\tilde L_t$ uniformly in $t$, and it is not clear whether such a function exists. To solve this problem, we will work on small intervals of time.

\bigskip

Let $T >0$. In the following, we will call $\omega_i$, $i\in\mathbb N$, several constants that depend on $T$ but not on $t\in[0,T]$. Since $t\mapsto\varepsilon_t$ and $t\mapsto \theta(\varepsilon_t)$ are locally Lipschitz functions, and from Assumption \ref{HypoPotentiel}, there exists $\omega_1 >0$ such that for all $t,s\in[0,T]$,
\[|\po \tilde L_t - \tilde L_s \pf f(x,y)| \leqslant \omega_1 (1 + |x| + |y|)|\na_y f|(x,y)| t-s|.\]
Let $R_\varepsilon$ be such as defined in Lemma \ref{LemmeLyapunov1}. Then it is easy to see that there exist $\omega_i$, $i=2,3,4$, such that for all $t\in[0,T]$,
\begin{eqnarray*}
\tilde L_{t} R_{\varepsilon_t}(x,y) & \leqslant & - \omega_2 (1+|x| + |y|)^2 + \omega_3,\\
|\na_y R_\varepsilon(x,y)| & \leqslant & \omega_4(1+|x|+|y|),
\end{eqnarray*}
and thus
\begin{eqnarray*}
\tilde L_{s} R_{\varepsilon_t}(x,y) & \leqslant & - \omega_2 (1+|x| + |y|)^2 + \omega_3 + \omega_1 \omega_4(1 + |x| + |y|)^2| t-s|.
\end{eqnarray*}
In particular, if $|t-s| \leqslant \frac{\omega_2}{2\omega_1\omega_4}$, outside the compact set $\mathcal K =\{(1+|x| + |y|)^2 \leq \frac{2\omega_3}{\omega_2}\}$,
\begin{eqnarray*}
L_{\varepsilon_s} R_{\varepsilon_t}(x,y) & \leqslant & 0 .
\end{eqnarray*}
Let $t_0=0 < t_1 < \dots < t_N = T$ be such that $|t_n - t_{n+1}| < \frac{\omega_2}{2\omega_1\omega_4}$, and let
\[V_n(x,y) = \sqrt{1+ R_{\varepsilon_{t_n}}(x,y) - \underset{\R^2}\min R_{\varepsilon_{t_n}}}.\]
Then $V_n$ is a Lipschitz function and if $t\in[t_n,t_{n+1}]$, outside $\mathcal K$,
\begin{eqnarray*}
\tilde L_t V_n & = & \frac{\tilde L_t R_{\varepsilon_{t_n}}}{2 V_n} - \frac{\Gamma_{\tilde L_t,(.)^2}\po R_{\varepsilon_{t_n}}\pf}{2 V_n^3}\\
& \leq & 0.
\end{eqnarray*}
 When $n$ and $\kappa$ are fixed, there exists $p_0$ such that $\mathcal K \subset \{V_n \leq p_0 - \frac\pi\kappa\}$.

We are now ready to define our truncation:
\[\eta_{\kappa,p,n} = l(V_n -p)\]
(recall that $\kappa$ intervenes in the definition of $l$). In this section, we have proved the following:

\begin{lem}\label{Lemeta}
For all $n\in  \llbracket 0, N-1\rrbracket $, $\kappa>0$,  $t\in[t_n,t_{n+1}]$ and  $p \geq p_0$, we have
\[\tilde L_t \eta_{\kappa,p,n} \geqslant -\frac{3\kappa^2}{4\pi} \| \na_y V_n\|^2 \eta_{\kappa,p,n}.\]
\end{lem}

\subsection{End of the proof of Proposition \ref{PropH}}

Bringing together Lemma \ref{LemKtilde},\ref{LemK} and \ref{Lemeta} and the computations of Section \ref{SectionGamma}, we get:

\begin{lem}\label{LemIH}
There exist $\omega_5,\omega_6$ (depending only on $T$) such that for all $n\in  \llbracket 0, N-1\rrbracket $, $\kappa>0$,  $t\in[t_n,t_{n+1}]$ and  $p \geq p_0$, writing $\eta = \eta_{\kappa,p,n}$,
\begin{eqnarray*}
\partial_t  I_{\eta} (t) & \leq &  \omega_5 \po I_{\eta}(t) + 1 \pf \\
\\
\partial_t  H_{\eta} (t) & \leq & \omega_6\kappa^2\po \frac1e + H_\eta(t)\pf - I_\eta(t) + |\varepsilon_t'|\xi(\varepsilon_t)\po H_\eta(t) + 1 + \mathbb E(X_t^2 + Y_t^2)\pf .
\end{eqnarray*}
\end{lem}

\begin{proof}
Let $\omega_6 = \frac{3}{4\pi} \underset{0\leq n \leq N-1}\max \| \na_y V_n\|_\infty$ and $\omega_7 = \underset{\varepsilon_0\geq \varepsilon\geq \varepsilon_T}\max \po \frac{\theta\po \varepsilon\pf}{\varepsilon}\|\na_x U \|_\infty + 1 + \frac{1}{\theta\theta\po \varepsilon \pf}\pf$. From Corollary \ref{CorGamma}, for all $t\in[0,T]$,
\[\Gamma_{L^*_{\varepsilon_t},\Phi_2} \geq  -\omega_7 \Phi_2.\]
Hence, from Lemma \ref{LemKtilde},\ref{LemK}, and \ref{Lemeta}, by taking $\kappa = 1$, $n\in\llbracket 0,N-1\rrbracket$ and $p\geq p_0$, we get
\begin{eqnarray*}
\partial_t I_{\eta}(t) & \leq & \omega_6 \po \frac1e + I_\eta(t)\pf + 2 \omega_7 I_\eta(t) +|\varepsilon_t'|\xi(\varepsilon_t)\po I_\eta(t) + 1 + \mathbb E\po X_t^2 + Y_t^2\pf\pf\\
& \leq & \omega_5 (I_\eta(t) + 1)
\end{eqnarray*}
for some $\omega_5$ (the moments are uniformly bounded on $[0,T]$ according to Proposition \ref{PropMoment}). The case of $H_\eta$ is exactly the same.
\end{proof}

\begin{proof}[of Proposition \ref{PropH}]
From Lemma \ref{LemIH}, for all $n\in  \llbracket 0, N-1\rrbracket $, for all $t\in[t_n,t_{n+1}]$, and for all $p$ large enough,
\begin{eqnarray*}
I_{\eta_{1,p,n}}(t) + 1 & \leq & e^{\omega_5 (t-t_n)} \po I_{\eta_{1,p,n}}(t_n)+1\pf.
\end{eqnarray*}
When $p$ goes to infinity, the Fatou Lemma yields $I(t) + 1  \leq  e^{\omega_5 (t-t_n)} \po I(t_n)+1\pf$, and thus, $I(t) + 1  \leq  e^{\omega_5 t} \po I(0)+1\pf$ for all $t\in[0,T]$. Finally, $I(0) < \infty$ since
\begin{eqnarray*}
|\na \sqrt{h_0}|^2 \mu_{\varepsilon_0} & = & \left|\na \sqrt{h_0 \mu_{\varepsilon_0}} - \sqrt{h_0}\na \sqrt{\mu_{\varepsilon_0}}\right|^2\\
& \leq & 2|\na \sqrt{m_0}|^2 + 2 h_0 |\na \sqrt{\mu_{\varepsilon_0}}|^2\\
& \leq & 2|\na \sqrt{m_0}|^2 + \frac{1}{2} m_0 |\na \ln \mu_{\varepsilon_0}|^2,
\end{eqnarray*}
which is integrable according to Assumption \ref{Hypo}.\ref{HypoLoiInit}.

\bigskip

Note that, by the log-Sobolev inequality \eqref{EqLogSob}, this implies that $t\mapsto H_t$ is also finite and locally bounded. Integrating the second part of Lemma \ref{LemIH} between times $s,t\in[t_n,t_{n+1}]$ yields, for $p$ large enough and writing $\eta = \eta_{\kappa,p,n},$
\begin{eqnarray*}
H_{\eta}(t) - H_{\eta}(s) & \leq & \int_s^t \omega_6\kappa^2\po \frac1e + H_\eta(u)\pf - I_\eta(u) + |\varepsilon_u'|\xi(\varepsilon_u)\po H_\eta(u) + \mathbb E(X_u^2+ Y_u^2)\pf \dd u.
\end{eqnarray*}
Again, we let $p$ go to infinity and use the Fatou Lemma to get
\begin{eqnarray*}
H(t) - H(s) & \leq & \int_s^t \omega_6\kappa^2\po \frac1e + H(u)\pf - I(u) + |\varepsilon_u'|\xi(\varepsilon_u)\po H(u) + \mathbb E(X_t^u + Y_t^u)\pf \dd u.
\end{eqnarray*}
Then, let $\kappa$ go to 0: the result does not depend on $n$, or even on $T$ any more, and thus it is true for any $t,s>0$.
\end{proof}

\section{Conclusion}\label{SectionConclusion}

We keep here the notations of the previous section. As announced in Section \ref{SubSectionSketch2}, the differential inequality satisfied by the distorted entropy implies it goes to zero:

\begin{lem}\label{LemHInequaDiff}
Under Assumption \ref{Hypo}, for  any $\alpha>0$,  there exists $B>0$ such that
\begin{eqnarray*}
H(t) & \leqslant & B \po \frac1t\pf^{1 -\frac{E_*}{E} - \alpha}.
\end{eqnarray*}
\end{lem}

\begin{proof} 
The log-Sobolev inequality \eqref{EqLogSob} implies $ I(t)\geq \xi_2(\varepsilon_t) e^{-\frac{E_*}{\varepsilon_t}} H(t)$ for all $t>0$, where $\xi_2$ is a sub-exponential function, so that Proposition \ref{PropH} becomes
\[H'(t) \leq - \xi_2(\varepsilon_t) e^{-\frac{E_*}{\varepsilon_t}} H(t)  + |\varepsilon_t'|\xi(\varepsilon_t)\po H(t) + 1+ \mathbb E(X_t^2 + Y_t^2)\pf .\]
Since $\xi$ is sub-exponential and $\varepsilon_t \ln t$ is bounded below by a positive constant for $t$ large enough (from the slow cooling assumption), then $\forall \alpha >0$,
\[t^{-\alpha} \xi\po \varepsilon_t \pf \ = \ e^{-\ln t\po \alpha - \frac{1}{\varepsilon_t \ln t} \po \varepsilon_t \ln \xi\po \varepsilon_t\pf\pf\pf} \ \underset{t\rightarrow+\infty}\longrightarrow\ 0,\]
and similarly for $\xi_2$. Moreover,  for $t$ large enough, $|\varepsilon_t'| \leq \frac{ \varepsilon_0^2}{Dt}$ which, together with Proposition \ref{PropMoment}, means that, for any $\alpha>0$, 
\[\beta(t)\ :=\ | \varepsilon'_t| \xi(\varepsilon_t) \po 1+ \mathbb E(X_t^2 + Y_t^2)\pf \ = \ \underset{t\rightarrow+\infty}o(t^{-1+\alpha}).\]
Similarly, for any $\alpha>0$, there exists $c>0$ such that, for $t$ large enough,
\[b(t)\ := \ \xi_2(\varepsilon_t) e^{-\frac{E_*}{\varepsilon_t}} \ \geqslant \  c \po \frac{1}{t}\pf^{\frac{E_*}{E}+\alpha}, \]
and $\beta = o(b)$ as $t\rightarrow \infty$. Hence, since $H>0$, $(b - \beta) H \geqslant \frac12 b H$ for large times, which means there exist  $t_0,c_1,c_2$ such that for  all $t\geqslant t_0$,
\[H'(t) \ \leqslant \ - c_1 \po \frac{1}{t}\pf^{\frac{E_*}{E}+\alpha} H(t) + c_2 \po \frac1t\pf^{1-\alpha}.\]

As proved in \cite[Lemma 6]{Miclo92}, this implies that $H$ goes to zero. For the sake of completeness, and to precise a speed of convergence toward zero, we recall here this short argument: if $\alpha \leqslant \frac12\po 1-\frac{E_*}{E}\pf$, for $t_0$ large enough and for all $t\geq s \geq t_0$, 
\begin{eqnarray*}
\partial_s \po H(s) - \frac{2c_2}{c_1} \po \frac{1}{s}\pf^{1 -\frac{E_*}{E} -2 \alpha} \pf  & \leqslant & - c_1 \po \frac{1}{s}\pf^{\frac{E_*}{E} + \alpha}\po H(s) - \frac{2c_2}{c_1} \po \frac{1}{s}\pf^{1 -\frac{E_*}{E} -2 \alpha} \pf,
\end{eqnarray*}
or in other words
\begin{eqnarray*}
 \partial_s \po \po H(s) - \frac{2 c_2}{c_1} \po \frac{1}{s}\pf^{1 -\frac{E_*}{E} -2 \alpha}\pf e^{c_1\int_{t_0}^s\po \frac{1}{u}\pf^{\frac{E_*}{E} + \alpha }\dd u }\pf & \leqslant & 0.
 \end{eqnarray*}
As a consequence,
\begin{eqnarray*} 
H(t) & \leq & \frac{2c_2}{c_1} \po \frac{1}{t}\pf^{1 -\frac{E_*}{E} -2 \alpha} + H(t_0) e^{- \frac{c_1}{v}(t^v - t_0^v)},
\end{eqnarray*}
where $v = 1 - \frac{E_*}{E} - \alpha >0$, which concludes.
\end{proof}

 We can now conclude the proof of our main result:
 
 \begin{proof}[Proof of Theorem \ref{TheoPrincipal}]
  For all $t\geq0$, let $(\tilde X_t,\tilde Y_t)$ be a random variable with law $\mu_{\varepsilon_t}$. Then for all $\delta>0$,
\begin{eqnarray*}
\mathbb P\po U(X_t) > \min U + \delta \pf & \leqslant & \mathbb P\po U(\tilde X_t) > \min U + \delta \pf + \| h_t - 1 \|_{L^1(\mu_{\varepsilon_t})}.
\end{eqnarray*}
 By Pinsker's inequality,
\[\| h_t - 1 \|_{L^1(\mu_{\varepsilon_t})} \ \leqslant\ \sqrt{2 \text{Ent}_{\mu_{\varepsilon_t}}(h_t)} \leq \sqrt{2 H(t)}.\]
From Lemmas \ref{LemHInequaDiff} and \ref{LemMassConcentration}, for all $\delta,\alpha>0$, there exists $c>0$ such that
\begin{eqnarray*}
\mathbb P\po U(X_t) > \min U + \delta \pf & \leqslant & c\po \po \frac1t\pf^{\frac{1 -\frac{E_*}{E} - \alpha}2} + e^{-\frac{\delta-\alpha}{\varepsilon_t}}\pf.
\end{eqnarray*}
The rhs goes to zero as $t$ goes to infinity, and is of order $\po\frac1 t\pf^{\frac{\min\po\delta,\frac{E - E_*}{2}\pf-\alpha}{E}}$ if $\partial_t \po \frac{1}{\varepsilon_t}\pf = \frac{1}{Et} $ for $t$ large enough.
 \end{proof}

\subsection*{Acknowledgments}

The author would like to thank Laurent Miclo, who initiated this work, for numerous fruitful discussions. This work has been supported by ANR STAB.

\bibliographystyle{plain}
\bibliography{biblio}

\end{document}